\documentclass[preprint,12pt]{elsarticle}

\usepackage{amsthm,amsmath,amsfonts,amssymb,amscd,mathrsfs}
\usepackage{txfonts}
\usepackage{supertabular,soul}
\usepackage[usenames,dvipsnames]{xcolor}
\usepackage{tikz, graphicx,color,geometry}
\usepackage{multirow}
\usetikzlibrary{arrows}
\usepackage{blindtext,ulem}

\usepackage[pdftex,
            pdfauthor={Francis},
            pdftitle={Equitable Decompositions of Graphs using Arbitrary Automorphisms},
            pdfsubject={Equitable Partitions, Spectral Graph Theory},
            pdfkeywords={Equitable Partitions, Spectral Graph Theory}]{hyperref}

\usepackage{bbm} 
\usepackage{hyperref}
\usepackage{yfonts}
\usepackage{eucal}
\usepackage{overpic}
\usetikzlibrary{calc}
\usepackage{enumitem}

\newcommand{\comment}[1]{}

\def\m{\medskip}

\def \t{\theta}

\newcommand{\defital}{\textit}

\newcommand{\Aut}{\text{Aut}}

\newcommand{\cT}{\mathcal{T}}

\newcommand{\hM}{\hat{M}}
\newcommand{\tM}{\tilde{M}}

\newcommand{\tp}{\tilde{\phi}}

\newcommand{\tcT}{\tilde{\mathcal{T}}}
\renewcommand{\so}{\mathscr{O}}
\newcommand{\w}{\omega}

\newtheorem{thm}{Theorem}[section]

\newtheorem{lem}[thm]{Lemma}

\newtheorem{prop}[thm]{Proposition}

\theoremstyle{definition}
\newtheorem{remark}[thm]{Remark}

\newtheorem{defn}[thm]{Definition}

\newtheorem{example}[thm]{Example}

\theoremstyle{remark}

\providecommand*{\propertyautorefname}{Property}

\let\oldmarginpar\marginpar
\renewcommand\marginpar[1]{\oldmarginpar[\raggedleft\footnotesize #1]%
{\raggedright\footnotesize #1}}

\renewcommand{\emph}{\textit}

\begin{document}
\begin{frontmatter}
\date{\today}

\title{General Equitable Decompositions for Graphs with Symmetries}


\author[amanda]{Amanda Francis}

\address[amanda]{{Mathematical Reviews, American Mathematical Society, Ann Arbor, MI 48103, USA, aefr@umich.edu}}

\author[dallas]{Dallas Smith}

\address[dallas]{Department of Mathematics, Brigham Young University, Provo, UT 84602, USA, dallas.smith@mathematics.byu.edu }

\author[ben]{Benjamin Webb}

\address[ben]{Department of Mathematics, Brigham Young University, Provo, UT 84602, USA, bwebb@mathematics.byu.edu}


\begin{abstract}
Using the theory of equitable decompositions it is possible to decompose a matrix $M$ appropriately associated with a given graph. The result is a collection of smaller matrices whose collective eigenvalues are the same as the eigenvalues of the original matrix $M$. This is done by decomposing the matrix over a graph symmetry. Previously it was shown that a matrix can be equitably decomposed over any uniform, basic, or separable automorphism. Here we extend this theory to show that it is possible to equitably decompose a matrix over any automorphism of a graph, without restriction. Moreover, we give a step-by-step procedure which can be used to generate such a decomposition.
We also prove under mild conditions that if a matrix $M$ is equitably decomposed the resulting divisor matrix, which is the divisor matrix of the associated equitable partition, will have the same spectral radius as the original matrix $M$.
\end{abstract}




\end{frontmatter}

\section{Introduction}

Spectral graph theory considers the relationship between the structure of a graph $G$ and its spectral properties. These {spectral properties} are typically the eigenvalues and eigenvectors of a matrix $M$ associated with the graph. The particular structures we consider here are graph symmetries. A graph {symmetry} is a permutation $\phi$ of the graph's vertices $V(G)$ that preserves the graph's (weighted) adjacencies (an \emph{automorphism} of $G$).

In our previous work \cite{BFW,FSSW} we showed that if a graph $G$ has a particular type of automorphism $\phi$ then it is possible to decompose any matrix $M$ that respects the structure of $G$ into a number of smaller matrices $M_{\phi},B_1,\dots,B_{k-1}$ in a way that preserves the eigenvalues of $M$, i.e.,
\begin{equation*}\label{eq:first}
\sigma(M)=\sigma(M_{\phi})\cup\sigma(B_1)\cup\cdots\cup\sigma(B_{k-1}),
\end{equation*}
so that the collective eigenvalues of the smaller matrices are the eigenvalues of $M$.
This method of decomposing a matrix over a graph symmetry is referred to as an \emph{equitable decomposition} due to its connection with the theory of equitable partitions.

An equitable partition is a particular partition of the graph's vertices that can be used to create, from the graphs adjacency matrix $A$,  a smaller matrix $A_{\phi}$ whose eigenvalues are a subset of the spectrum of $A$ (see Definition \ref{def:EP}). As it turns out, the orbits of any graph symmetry give an equitable partition, although the converse does not hold (see Theorem 9.3.3 of \cite{Godsil} and Theorem 3.9.5 of \cite{cvet}).

In \cite{BFW} both equitable partitions and equitable decompositions are defined for matrices beyond the adjacency matrix of a graph, including the  various Laplacian matrices, distance matrices, etc. (see Proposition 3.4).  This class of matrices, referred to as \emph{automorphism compatible} matrices, are those matrices that respect the symmetries of a graph $G$ ($M_{\phi(i), \phi(j)} = M_{i,j}$ for any $\phi \in \Aut(G)$). Importantly, the matrix $M_{\phi}$ in the resulting equitable decomposition is the same as the matrix that results from an equitable partition of $G$, using the orbits of $\phi$.

The particular types of automorphisms for which equitable decompositions are defined are the so-called uniform, basic, and separable automorphisms. A \emph{uniform} automorphism $\phi$ is an automorphism in which all orbits of the automorphism have the same cardinality (see \cite{BFW}). A \emph{basic} automorphism $\phi$ is an automorphism for which all orbits of size greater than one have the same cardinality (see \cite{BFW}). A \emph{separable} automorphism is an automorphism whose order is the product of distinct primes (see \cite{FSSW}).

Since many graph automorphisms are neither uniform, basic, nor separable, a natural question is whether an automorphism compatible matrix $M$ can be equitably decomposed over other automorphisms. The major contribution of this paper is answering this question in the affirmative (see Theorem \ref{thm:gendecomp}). That is, the theory of equitable decompositions can be extended to \emph{any} automorphism of a graph $G$.

As an intermediate step to proving this result, we show in Theorem \ref{thm:fullprimedecomp} that an automorphism compatible matrix $M$ can be equitably decomposed over any prime-powered automorphism (an automorphism $\phi$ with $|\phi| = p^N$ for some prime $p$ and  $N\geq 1$). We use this result to give an algorithm for decomposing $M$ over a general automorphism $\phi$ of order
\[
|\phi|=p_1^{N_1}p_2^{N_2}\cdots p_{h}^{N_{h}} \ \ \text{for distinct primes} \ \ p_1,p_2,\dots, p_{h};
\]
by sequentially decomposing over $h$ prime-powered automorphisms, corresponding to $p_1, \ldots, p_h$. The result is the equitable decomposition
\[
M_{\phi}\oplus B_1\oplus \cdots \oplus B_{k}
\]
of $M$ over $\phi$  where, as in any equitable decomposition, the collective eigenvalues of the smaller matrices are the eigenvalues of the original matrix $M$.


Last, we demonstrate that for any automorphism compatible matrix $M$ and automorphism $\phi$ the divisor matrix $M_{\phi}$ and the matrix $M$ have the same spectral radius if $M$ is both nonnegative and irreducible (see Proposition \ref{lem:3}). As this holds for any automorphism $\phi$ this extends the result of \cite{FSSW} in which this  was shown to hold for both basic and separable automorphisms.

It is worth emphasizing that an equitable decomposition of $M$ does not require any knowledge of the matrix' eigenvalues or eigenvectors, as opposed to a spectral decomposition (diagonalization) of $M$. Only knowledge of a symmetry is needed. {The surprising result is that  if an automorphism involves only part of the graph (a \emph{local} symmetry) this local information can be used to determine properties of the associated eigenvalues
and  {eigenvectors},
which in general depend on the entire graph structure.

This method of using local symmetries to determine spectral properties of a graph is perhaps most useful in analyzing the spectral properties of real-world networks since many of these have a high degree of symmetry \cite{MacArthur} when compared, for instance, to randomly generated graphs \cite{Aldous2000,Newman10,Strogatz03,Watts99}. From a practical point of view, the large size of these networks limit our ability to quickly compute their associated eigenvalues and eigenvectors, which are used in a number of standard network metrics and algorithms \cite{Newman10}. However, their high degree of symmetry suggests that it may be possible to effectively estimate a network's spectral properties by equitably decomposing the network over local symmetries, which is a potentially much more feasible task (see Examples 5.3 and 5.4 from \cite{FSSW}).

This paper is organized as follows. In Section \ref{sec:ep} we summarize the theory of equitable decompositions found in \cite{BFW}. In Section \ref{sec:pp} we describe how the theory of equitable decompositions can be extended to prime-powered automorphisms. We use this in Section \ref{GED} to extend the theory of equitable decompositions to any automorphism. We also present algorithms describing how an automorphism compatible matrix can be equitable decomposed over any prime-powered automorphism (in Section \ref{sec:pp} ) and general automorphism (in Section \ref{GED}). In Section \ref{sec:extapp} we prove that
the original matrix and its divisor matrix have the same spectral radius. Section {\ref{sec:conc}} contains some closing remarks including a few open questions regarding equitable decompositions.

\section{Graph Symmetries and Equitable Decompositions}\label{sec:ep}

The main objects considered in this paper are matrices and graphs. A \emph{graph} $G$ is made up of a finite set of vertices $V=V(G)=\{1,\dots,n\}$ and a finite set of edges $E=E(G)$. A graph can be \emph{undirected}, meaning that each edge $(i,j)\in E$ can be thought of as an unordered pair so that $(i,j)=(j,i)$. A graph is \emph{directed} when each edge is \emph{directed}, in which case $(i,j)\in E$ is an ordered pair where it is not necessarily true that if $(i,j)\in E$ that $(j,i)\in E$. In both directed and undirected graphs, a \emph{loop} is an edge with only one vertex, i.e. $(i,i)\in E$. A \emph{weighted graph} is a graph, either directed or undirected, in which each edge $(i,j)\in E$ is assigned a numerical weight $w(i,j)$.

As a major goal of this paper is to understand the relationship between the structure of a graph, specifically its symmetries, and its spectral properties we need a way to associate a matrix with a graph. In practice there are a number of matrices that may be associated with a given graph $G$. One of the most common is the adjacency matrix $A=A(G)$ given by 
\[
A_{i,j}=
\begin{cases}
1 &\text{if} \ \ (i,j)\in E(G)\\
0 &\text{otherwise}.
\end{cases}
\]


For an $n\times n$ matrix $M$ associated with a graph $G$ we let $\sigma(M)$ denote the \emph{eigenvalues} of $M$. For us $\sigma(M)$ is a multiset with each eigenvalue in $\sigma(M)$ listed according to its multiplicity. To simplify our discussion we will often refer to $\sigma(M)$ as the eigenvalues of the graph $G$ when the context makes it clear that the matrix $M$ is associated with $G$.

As previously mentioned the specific type of structures we consider here are graph symmetries. 
Such graph symmetries are formally described by the graph's set of automorphisms.

\begin{defn} \textbf{(Graph Automorphism)}
An \emph{automorphism} $\phi$ of an unweighted graph $G$ is a bijection $\phi: V(G) \to V(G)$ such that $(i,j)$ is in $E(G)$ if and only if $(\phi(i), \phi(j))$ is in $E(G)$. For a weighted graph $G$, if $w(i,j) = w(\phi(i), \phi(j))$ for each pair of vertices $i$ and $j$, then $\phi$ is an automorphism of $G$.
\end{defn}

The set of all automorphisms of $G$ is a group, denoted by $\Aut(G)$. The \emph{order} $|\phi|$ of $\phi\in\Aut(G)$ is the smallest positive integer $\ell$ such that $\phi^\ell$ is the identity map on $V(G)$. For a graph $G$ with automorphism $\phi$, we define the relation $\sim$ on $V(G)$ by $u \sim v$ if and only if $v = \phi^j(u)$ for some nonnegative integer $j$. Then $\sim$ is an equivalence relation on $V(G)$, and the equivalence classes are called the \emph{orbits} of $\phi$. We denote the orbit associated with the vertex $i$ by $\so_\phi(i)$ whose {length} $|\so_\phi(i)|$ is the size of the orbit.

Here, as in \cite{BFW} and \cite{FSSW} we consider those matrices $M$ associated with a graph $G$ whose structure mimics the symmetries of the graph.

\begin{defn}\label{def:autocomp}\textbf{(Automorphism Compatible)}
Let $G$ be a graph on $n$ vertices. An $n \times n$ matrix $M$ is \emph{automorphism compatible} on $G$ if, given any automorphism $\phi$ of $G$ and any $i, j \in \{1, 2, \ldots, n\}$,
$M_{\phi(i), \phi(j)} = M_{i, j}$.
\end{defn}

Some of the most well-known matrices that are associated with a graph are automorphism compatible. This includes the adjacency matrix, combinatorial Laplacian matrix, signless Laplacian matrix, normalized Laplacian matrix, and distance matrix of a simple graph. Additionally, the weighted adjacency matrix of a weighted graph is automorphism compatible. (See Proposition 3.4, \cite{BFW}.)

Since the symmetries of a graph can be found in the structure of any automorphism compatible matrix it is also possible to talk about the symmetries or automorphisms of a matrix. That is, $\phi$ is an automorphism of an $n\times n$ matrix $M$, as in Definition \ref{def:autocomp}, if $M_{\phi(i) ,\phi(j)} = M_{i, j}$ for any $i, j \in \{1, 2, \ldots, n\}$. We let $\Aut(M)$ denote the automorphism group of $M$.

Here we state a generalization of the theory of equitable partitions, originally given in \cite{BFW}.

\begin{defn}\label{def:EP}\textbf{(Equitable Partition)}
An \emph{equitable partition} of a graph $G$ and a matrix $M$ associated with $G$, is a partition $\pi$ of $V(G)$ into $V(G) = V_1 \cup \ldots \cup V_k$, which has the property that for all $i$, $j \in \{1, 2, \ldots, k\}$
\begin{equation}\label{eq:1}
\sum_{t \in V_j} M_{s,t} = D_{i,j}
\end{equation}
is a constant $D_{i,j}$ for any  $s \in V_i$. The $k \times k$ matrix $M_\pi = D
$ is called the \defital{divisor matrix} of $M$ associated with the equitable partition $\pi$.
\end{defn}

Note for {simple graphs} a partition $\pi$ is an equitable partition if and only if any vertex $\ell \in V_i$ has the same number of neighbors in $V_j$ for all $i,j \in \{1, \ldots, k\}$ (for example, see p. 195-6 of  \cite{Godsil}).
Further, if $\phi$ is any automorphism of $G$ and $M$ is a matrix associated with $G$ that is compatible with $\phi$, the orbits of $\phi$ form an equitable partition of $V(G)$ (see Proposition 3.2, \cite{BFW}).

In \cite{BFW} the following special case of an equitable partition is considered. Suppose the nontrivial orbits of $\phi$ are all the same length, in which case we refer to $\phi$ as a \emph{basic automorphism}. Then $\phi$ can be used to fully decompose $M$ into a number of smaller matrices, one of which is the divisor matrix $D$ associated with the equitable partition induced by $\phi$ (see Theorem 4.4, \cite{BFW}). This decomposition is called an \emph{equitable decomposition}, which we briefly describe here.
We form a \textit{semi-transversal} $\cT_0$ of the orbits of the basic automorphism $\phi$  by choosing one vertex from each orbit of \emph{size} $k$. Further we define
\begin{equation}\label{eq:tranversal}
\cT_\ell = \{\phi^\ell(v) \ | \ v \in \cT_0\}
\end{equation}
for $\ell = 0,1, \ldots, k-1$ to be the $\ell$th power of $\cT_0$ and we let $M[\mathcal{T}_i,\mathcal{T}_j]$ be the submatrix of $M$ whose rows are indexed by $\mathcal{T}_i$ and whose columns are indexed by $\mathcal{T}_j$. We let $\cT_f$ denote the vertices fixed by $\phi$, which are
\[
\cT_f=\{v\in V(G) \ |\phi(v)=v\}.
\]
These definitions allow us to decompose an automorphism compatible matrix $M$ in the following way.

\begin{thm}\label{thm:2} \textbf{(Basic Equitable Decomposition) \cite{BFW}}
Let $G$ be a graph on $n$ vertices, let $\phi$ be a basic automorphism of $G$ of size $k>1$, let $\cT_0$ be a semi-transversal of the $k$-orbits of $\phi$, let $\cT_f$ be the vertices fixed by $\phi$, 
 and let $M$ be an automorphism compatible matrix on $G$.  Set $F = M[\cT_f,\cT_f]$, $H = M[\cT_f,\cT_0]$, $L=M[\cT_0,\cT_f]$, $M_m = M[\cT_0, \cT_m]$, for $m = 0, 1, \ldots, k-1$, $\omega = e^{2 \pi i /k}$, and
\begin{equation}\label{eq:B}
B_j = \sum_{m=0}^{k-1} \omega^{jm} M_m,  \ \ j = 0, 1, \ldots, k-1.
\end{equation}
Then there exists an invertible matrix $S$ that can be explicitly constructed such that
\begin{equation}\label{eq:spectrum2}
S^{-1}MS=M_{\phi}\oplus B_1\oplus B_2\oplus\cdots B_{k-1}
\end{equation}
where $M_\phi=\left[\begin{array}{rr} F & kH \\ L & B_0 \end{array}\right]$ is the divisor matrix associated with $\phi$. Thus
$${\sigma(M) = \sigma\left(M_\phi \right)
\cup \sigma(B_1) \cup \sigma(B_2) \cup \cdots \cup \sigma(B_{k-1}).}$$
\end{thm}


In \cite{FSSW} these results were extended beyond basic automorphisms to \emph{separable} automorphisms $\phi$ (where $|\phi |=p_1 p_2\dots p_k$ a product of distinct primes). This can be done by using a series of decompositions, one for each prime.

One might naively believe that we could also use this method to decomposed a matrix with an automorphism whose order is $p^N$, by the process  outlined in \cite{FSSW}.  The following is an example showing this method does  not work in general and demonstrating a  need for a more sophisticated method to decompose matrices over automorphisms of order $p^N$.

\begin{example}\label{ex:1}

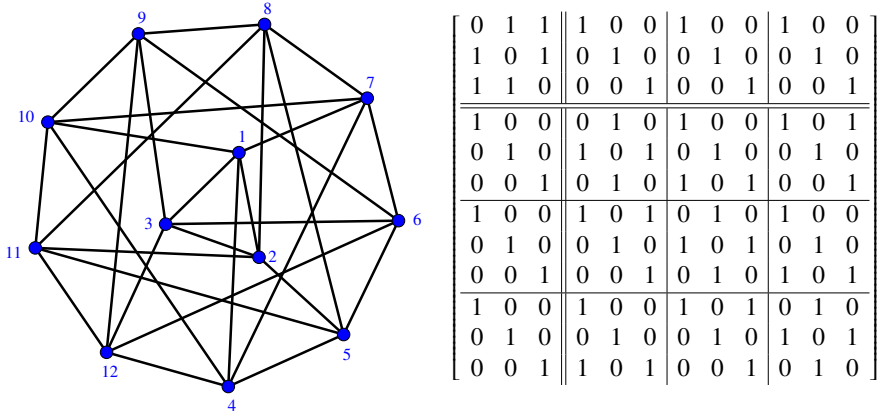
\begin{figure}[h]
\resizebox{.8\textwidth}{!}{
\begin{tabular}{cc}
\begin{minipage}[b]{.47\textwidth}

\begin{tikzpicture}[line cap=round,line join=round,>=triangle 45,x=1.0cm,y=1.0cm,scale = .85]
\draw [line width=1.2pt] (-0.6,3.62)-- (2.123558053023431,9.276851992900792);
\draw [line width=1.2pt] (-4.137198504789,8.807096459050374)-- (2.123558053023431,9.276851992900792);
\draw [line width=1.2pt] (-4.137198504789,8.807096459050374)-- (-0.6,3.62);
\draw [line width=1.2pt] (-2.9869038033291604,4.291334665910222)-- (-2.3623537280050293,10.538548456168009);
\draw [line width=1.2pt] (-2.3623537280050293,10.538548456168009)-- (2.735617079568619,6.874065329872934);
\draw [line width=1.2pt] (2.735617079568619,6.874065329872934)-- (-2.9869038033291604,4.291334665910222);
\draw [line width=1.2pt] (-4.383852592928873,6.339879046497543)-- (0.11021214116330502,10.724069405453623);
\draw [line width=1.2pt] (0.11021214116330502,10.724069405453623)-- (1.66,4.64);
\draw [line width=1.2pt] (1.66,4.64)-- (-4.383852592928873,6.339879046497543);
\draw [line width=1.2pt] (-4.383852592928873,6.339879046497543)-- (-4.137198504789,8.807096459050374);
\draw [line width=1.2pt] (-4.137198504789,8.807096459050374)-- (-2.3623537280050293,10.538548456168009);
\draw [line width=1.2pt] (-2.3623537280050293,10.538548456168009)-- (0.11021214116330502,10.724069405453623);
\draw [line width=1.2pt] (0.11021214116330502,10.724069405453623)-- (2.123558053023431,9.276851992900792);
\draw [line width=1.2pt] (2.123558053023431,9.276851992900792)-- (2.735617079568619,6.874065329872934);
\draw [line width=1.2pt] (2.735617079568619,6.874065329872934)-- (1.66,4.64);
\draw [line width=1.2pt] (1.66,4.64)-- (-0.6,3.62);
\draw [line width=1.2pt] (-0.6,3.62)-- (-2.9869038033291604,4.291334665910222);
\draw [line width=1.2pt] (-2.9869038033291604,4.291334665910222)-- (-4.383852592928873,6.339879046497543);
\draw [line width=1.2pt] (-1.8285211049496375,6.806067771006954)-- (-2.3623537280050293,10.538548456168009);
\draw [line width=1.2pt] (-1.8285211049496375,6.806067771006954)-- (2.735617079568619,6.874065329872934);
\draw [line width=1.2pt] (-1.8285211049496375,6.806067771006954)-- (-2.9869038033291604,4.291334665910222);
\draw [line width=1.2pt] (0.001571740633340254,6.156128255753188)-- (-1.8285211049496375,6.806067771006954);
\draw [line width=1.2pt] (-0.3918127028097297,8.208568830238764)-- (-1.8285211049496375,6.806067771006954);
\draw [line width=1.2pt] (-0.3918127028097297,8.208568830238764)-- (0.001571740633340254,6.156128255753188);
\draw [line width=1.2pt] (0.001571740633340254,6.156128255753188)-- (0.11021214116330502,10.724069405453623);
\draw [line width=1.2pt] (0.001571740633340254,6.156128255753188)-- (1.66,4.64);
\draw [line width=1.2pt] (0.001571740633340254,6.156128255753188)-- (-4.383852592928873,6.339879046497543);
\draw [line width=1.2pt] (-0.3918127028097297,8.208568830238764)-- (-4.137198504789,8.807096459050374);
\draw [line width=1.2pt] (-0.3918127028097297,8.208568830238764)-- (2.123558053023431,9.276851992900792);
\draw [line width=1.2pt] (-0.3918127028097297,8.208568830238764)-- (-0.6,3.62);
\begin{scriptsize}
\draw [fill=blue] (-0.6,3.62) circle (3.5pt);
\draw[color=blue] (-0.5286420744421019,3.24120280245417) node {$4$};
\draw [fill=blue] (1.66,4.64) circle (3.5pt);
\draw[color=blue] (1.7290425574920387,4.250319418242912) node {$5$};
\draw [fill=blue] (2.735617079568619,6.874065329872934) circle (3.5pt);
\draw[color=blue] (3.1065738590969693,6.890900378723) node {$6$};
\draw [fill=blue] (2.123558053023431,9.276851992900792) circle (3.5pt);
\draw[color=blue] (2.1908416867512948,9.585414382289507) node {$7$};
\draw [fill=blue] (0.11021214116330502,10.724069405453623) circle (3.5pt);
\draw[color=blue] (0.17260845517380546,11.039226455883457) node {$8$};
\draw [fill=blue] (-2.3623537280050293,10.538548456168009) circle (3.5pt);
\draw[color=blue] (-2.2903202342088935,10.851086069888945) node {$9$};
\draw [fill=blue] (-4.137198504789,8.807096459050374) circle (3.5pt);
\draw[color=blue] (-4.569102065429732,8.923615253030252) node {$10$};
\draw [fill=blue] (-4.383852592928873,6.339879046497543) circle (3.5pt);
\draw[color=blue] (-4.808553465786383,6.260686563647559) node {$11$};
\draw [fill=blue] (-2.9869038033291604,4.291334665910222) circle (3.5pt);
\draw[color=blue] (-2.9231560780086148,3.908245989161983) node {$12$};
\draw [fill=blue] (-1.8285211049496375,6.806067771006954) circle (3.5pt);
\draw[color=blue] (-2.1601064191334512,6.822485692906814) node {$3$};
\draw [fill=blue] (0.001571740633340254,6.156128255753188) circle (3.5pt);
\draw[color=blue] (0.26998642644952634,6.172546177653048) node {$2$};
\draw [fill=blue] (-0.3918127028097297,8.208568830238764) circle (3.5pt);
\draw[color=blue] (-0.32339801699354365,8.524986752138625) node {$1$};
\end{scriptsize}
\end{tikzpicture}
\end{minipage}&

\begin{minipage}[b]{.47\textwidth}

 $\left[\begin{array}{lll || lll | lll | lll}
0 & 1 & 1 & 1 & 0 & 0 & 1 & 0 &  0 & 1 & 0 & 0\\
1 & 0 & 1 & 0 & 1 & 0 & 0 & 1 & 0 &  0 &  1 &  0\\
1 & 1 & 0 & 0 & 0 & 1 & 0 & 0 & 1 & 0 & 0 & 1\\\hline \hline
1 & 0 & 0 & 0 & 1 & 0 & 1 & 0 & 0 & 1 & 0 & 1\\
0 & 1 & 0 & 1 & 0 & 1 & 0 & 1 & 0 & 0 & 1 &0\\
0 & 0 & 1 & 0 & 1 & 0 & 1 & 0 & 1 & 0 & 0 & 1\\\hline
1 & 0 & 0 & 1 &0 & 1 & 0 & 1 & 0 & 1 & 0 & 0\\
0 & 1 & 0 & 0 & 1 & 0 & 1 & 0 & 1 & 0 & 1 & 0\\
0 & 0 & 1 & 0 & 0 & 1 & 0 & 1 & 0 & 1 & 0 & 1\\\hline
1 & 0 & 0 & 1 & 0 & 0 & 1 & 0 & 1 & 0 & 1 & 0\\
0 & 1 & 0 & 0 & 1 & 0 & 0 & 1 & 0 & 1 & 0 & 1\\
0 & 0 & 1 & 1 & 0 & 1 & 0 & 0 & 1 & 0 & 1 & 0
\end{array}\right]$
\vspace{.5cm}

\end{minipage}
\end{tabular}}
\caption{The graph $G$ on 12 vertices with automorphism $\phi=(1,2,3)(4,5,6,7,8,9,10,11,12)$, and its adjacency matrix $M$. Here vertices are labeled blue and edge weights are labeled black, which will be our convention throughout the paper.}\label{ex:12}
\end{figure}

Consider the following graph and its  adjacency matrix in Figure \ref{ex:12}.
We attempt to follow the recursive method of equitable decompositions for separable automorphisms in \cite{FSSW} by first forming a new automorphism $\psi=\phi^3=(4,7,10)(5,8,11)\allowbreak (6,9,12)$. The first decomposition gives the following direct sum of smaller matrices with the associated digraphs found in Figure \ref{threegraphs}:
\begin{equation}\label{ex25}
\left[\begin{array}{lll | lll }
0 & 1 & 1 & 3 & 0 & 0 \\
1 & 0 & 1 & 0 & 3 & 0\\
1 & 1 & 0 & 0 & 0 & 3 \\\hline
1 & 0 & 0 & 2 & 1 & 1 \\
0 & 1 & 0 & 1 & 2 & 1 \\
0 & 0 & 1 & 1 & 1 & 2
\end{array}\right] \oplus
\left[\begin{array}{lll }
\w +\w^2& 1 & \w^2  \\
1 & \w + \w^2 & 1 \\
\w & 1 & \w + \w^2  \\
\end{array}\right]\oplus
\left[\begin{array}{lll }
\w +\w^2& 1 & \w  \\
1 & \w + \w^2 & 1 \\
\w^2 & 1 & \w + \w^2  \\
\end{array}\right].\end{equation}

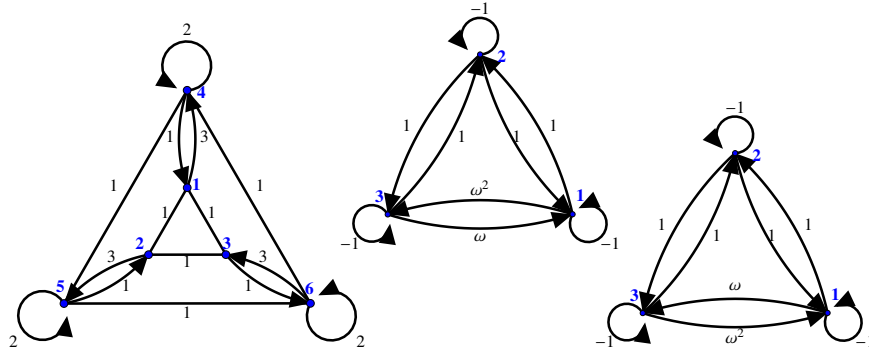
\begin{figure}[h!]
\begin{center}
\resizebox{.8\textwidth}{!}{
\begin{tabular}{c}
\begin{tikzpicture}[line cap=round,line join=round,>=triangle 45,x=1.0cm,y=1.0cm, scale = .4]
\draw [line width=1.2pt] (3.,2.)-- (1.42,-0.74);
\draw [line width=1.2pt] (1.42,-0.74)-- (4.582909606369362,-0.738320137979414);
\draw [line width=1.2pt] (4.582909606369362,-0.738320137979414)-- (3.,2.);
\draw [->, shift={(8.746530801304676,3.9975601362061384)},line width=1.2pt]  plot[domain=2.808109490764957:3.47613804206866,variable=\t]({1.*6.083819749803839*cos(\t r)+0.*6.083819749803839*sin(\t r)},{0.*6.083819749803839*cos(\t r)+1.*6.083819749803839*sin(\t r)});
\draw [->,shift={(-2.7619038286835558,3.9914478583805644)},line width=1.2pt]  plot[domain=-0.3327703910467745:0.33383261670080516,variable=\t]({1.*6.096343191097985*cos(\t r)+0.*6.096343191097985*sin(\t r)},{0.*6.096343191097985*cos(\t r)+1.*6.096343191097985*sin(\t r)});
\draw [->, shift={(2.574873153942068,-6.730834430650846)},line width=1.2pt]  plot[domain=1.7612336156607167:2.4286188147797056,variable=\t]({1.*6.1011334338111105*cos(\t r)+0.*6.1011334338111105*sin(\t r)},{0.*6.1011334338111105*cos(\t r)+1.*6.1011334338111105*sin(\t r)});
\draw [->,shift={(-3.190222966731918,3.2423685501799246)},line width=1.2pt]  plot[domain=4.902310947532458:5.570726790087551,variable=\t]({1.*6.09207805863029*cos(\t r)+0.*6.09207805863029*sin(\t r)},{0.*6.09207805863029*cos(\t r)+1.*6.09207805863029*sin(\t r)});
\draw [->,shift={(3.4409228960059806,-6.713365807470568)},line width=1.2pt]  plot[domain=0.7135095119397563:1.3819478161074714,variable=\t]({1.*6.083198533596579*cos(\t r)+0.*6.083198533596579*sin(\t r)},{0.*6.083198533596579*cos(\t r)+1.*6.083198533596579*sin(\t r)});
\draw [->,shift={(9.177088828763083,3.23416245566517)},line width=1.2pt]  plot[domain=3.854546186837524:4.524096448389289,variable=\t]({1.*6.073475173595705*cos(\t r)+0.*6.073475173595705*sin(\t r)},{0.*6.073475173595705*cos(\t r)+1.*6.073475173595705*sin(\t r)});
\draw [->,shift={(2.9803003568273967,6.988860475994255)},line width=1.2pt]  
  plot[domain=-1.5532143950535158:4.314098663792105,variable=\t]({1.*1.*cos(\t r)+0.*1.*sin(\t r)},{0.*1.*cos(\t r)+1.*1.*sin(\t r)}) ;
\draw [->,shift={(8.89586008462023,-3.249577222650493)},line width=1.2pt]    plot[domain=-3.6856496994623047:2.1335046414728542,variable=\t]({1.*1.*cos(\t r)+0.*1.*sin(\t r)},{0.*1.*cos(\t r)+1.*1.*sin(\t r)})  ;
\draw [->,shift={(-2.911791543957258,-3.2303056150093457)},line width=1.2pt]   plot[domain=0.512248008742023:6.265411724401079,variable=\t]({1.*1.*cos(\t r)+0.*1.*sin(\t r)},{0.*1.*cos(\t r)+1.*1.*sin(\t r)});
\draw [line width=1.2pt] (2.9978813827489854,5.9890150341745985)-- (-2.0401466673309936,-2.7401676660080625);
\draw [line width=1.2pt] (-2.0401466673309936,-2.7401676660080625)-- (8.040244329607463,-2.7319657083137923);
\draw [line width=1.2pt] (8.040244329607463,-2.7319657083137923)-- (2.9978813827489854,5.9890150341745985);
\begin{scriptsize}
\draw [fill=blue] (3.,2.) circle (4.5pt);
\draw[color=blue] (3.3797862649889464,2.1344492655878913) node {${\bf 1}$};
\draw [fill=blue] (1.42,-0.74) circle (4.5pt);
\draw[color=blue] (1.0607506144093768,-0.2652055658528048) node {${\bf 2}$};
\draw [fill=blue] (4.582909606369362,-0.738320137979414) circle (4.5pt);
\draw[color=blue] (4.621528751385895,-0.291008305975823) node {${\bf3}$};
\draw [fill=blue] (2.9978813827489854,5.9890150341745985) circle (4.5pt);
\draw[color=blue] (3.5797862649889464,5.901649323548554) node {${\bf4}$};
\draw [fill=blue] (-2.0401466673309936,-2.7401676660080625) circle (4.5pt);
\draw[color=blue] (-2.1903946410909225,-2.045594634341063) node {${\bf5}$};
\draw [fill=blue] (8.040244329607463,-2.7319657083137923) circle (4.5pt);
\draw[color=blue] (8,-2.0972001145870998) node {${\bf6}$};

\draw[color=black] (5.9797862649889464,2.101649323548554) node {$1$};
\draw[color=black] (0,2.101649323548554) node {$1$};
\draw[color=black] (3,-3.101649323548554) node {$1$};
\draw[color=black] (10.1,-4.301649323548554) node {$2$};
\draw[color=black] (-4.1,-4.301649323548554) node {$2$};
\draw[color=black] (3,8.5) node {$2$};
\draw[color=black] (3.7,4.1) node {$3$};
\draw[color=black] (2.3,4.1) node {$1$};
\draw[color=black] (-.1,-.8) node {$3$};
\draw[color=black] (0.5,-2) node {$1$};
\draw[color=black] (6.1,-.8) node {$3$};
\draw[color=black] (5.5,-2) node {$1$};
\draw[color=black] (2.1,1) node {$1$};
\draw[color=black] (4,1) node {$1$};
\draw[color=black] (3,-1) node {$1$};
\end{scriptsize}
\end{tikzpicture}

\hspace*{-.7cm}\begin{tikzpicture}[line cap=round,line join=round,>=triangle 45,x=1.0cm,y=1.0cm,scale = .3]
\draw [->, shift={(3.0130847985281797,-18.75756462853538)},line width=1.2pt]  plot[domain=1.2668226187643103:1.8763973439024366,variable=\t]({1.*16.79559923619067*cos(\t r)+0.*16.79559923619067*sin(\t r)},{0.*16.79559923619067*cos(\t r)+1.*16.79559923619067*sin(\t r)});
\draw [->, shift={(19.37392598872407,9.63923476036045)},line width=1.2pt]  plot[domain=3.3609073531733613:3.970731049757719,variable=\t]({1.*16.777930175868597*cos(\t r)+0.*16.777930175868597*sin(\t r)},{0.*16.777930175868597*cos(\t r)+1.*16.777930175868597*sin(\t r)});
\draw [->, shift={(-8.333612517864772,-6.380920499656055)},line width=1.2pt]  plot[domain=0.2192694750466929:0.8291836207048014,variable=\t]({1.*16.775519638174345*cos(\t r)+0.*16.775519638174345*sin(\t r)},{0.*16.775519638174345*cos(\t r)+1.*16.775519638174345*sin(\t r)});
\draw [->, shift={(14.35091983175103,-6.3818000228558365)},line width=1.2pt]  plot[domain=2.313319799013006:2.9229717803830764,variable=\t]({1.*16.790728012798304*cos(\t r)+0.*16.790728012798304*sin(\t r)},{0.*16.790728012798304*cos(\t r)+1.*16.790728012798304*sin(\t r)});
\draw [->, shift={(-13.389609820737382,9.628583917256625)},line width=1.2pt]  plot[domain=5.454838648429755:6.064638238145913,variable=\t]({1.*16.786790330466452*cos(\t r)+0.*16.786790330466452*sin(\t r)},{0.*16.786790330466452*cos(\t r)+1.*16.786790330466452*sin(\t r)});
\draw [->, shift={(2.9940079090559673,13.282142090759644)},line width=1.2pt]  plot[domain=4.407959240767685:5.017651628291141,variable=\t]({1.*16.794556327595302*cos(\t r)+0.*16.794556327595302*sin(\t r)},{0.*16.794556327595302*cos(\t r)+1.*16.794556327595302*sin(\t r)});
\draw [shift={(2.9803003568273967,6.988860475994255)},line width=1.2pt,->]  plot[domain=-1.5532143950535158:4.014098663792105,variable=\t]({1.*1.*cos(\t r)+0.*1.*sin(\t r)},{0.*1.*cos(\t r)+1.*1.*sin(\t r)}) ;
\draw [shift={(8.89586008462023,-3.249577222650493)},line width=1.2pt,->]    plot[domain=-3.6856496994623047:2.1335046414728542,variable=\t]({1.*1.*cos(\t r)+0.*1.*sin(\t r)},{0.*1.*cos(\t r)+1.*1.*sin(\t r)})  ;
\draw [shift={(-2.911791543957258,-3.2303056150093457)},line width=1.2pt, ->]    plot[domain=0.512248008742023:6.265411724401079,variable=\t]({1.*1.*cos(\t r)+0.*1.*sin(\t r)},{0.*1.*cos(\t r)+1.*1.*sin(\t r)})  ;
\begin{scriptsize}
\draw [fill=blue] (2.9978813827489854,5.9890150341745985) circle (4pt);
\draw[color=blue] (4.1797862649889464,5.901649323548554) node {${\bf 2}$};
\draw [fill=blue] (-2.0401466673309936,-2.7401676660080625) circle (4pt);
\draw[color=blue] (-2.4903946410909225,-1.745594634341063) node {${\bf 3}$};
\draw [fill=blue] (8.040244329607463,-2.7319657083137923) circle (4pt);
\draw[color=blue] (8.472674006886195,-1.6972001145870998) node {${\bf 1}$};

\draw[color=black] (-1,2.101649323548554) node {$1$};
\draw[color=black] (3,-4) node {$\w$};
\draw[color=black] (10.1,-4.301649323548554) node {$-1$};
\draw[color=black] (-4.1,-4.301649323548554) node {$-1$};
\draw[color=black] (3,8.5) node {$-1$};
\draw[color=black] (3,-1.3) node {$\w^2$};
\draw[color=black] (2,1.5) node {$1$};
\draw[color=black] (5,1.5) node {$1$};
\draw[color=black] (7,2.101649323548554) node {$1$};
\draw[color=black] (3,-10) node {$$};

\end{scriptsize}
\end{tikzpicture}

\hspace*{-.7cm}\begin{tikzpicture}[line cap=round,line join=round,>=triangle 45,x=1.0cm,y=1.0cm,scale = .3]
\draw [->, shift={(3.0130847985281797,-18.75756462853538)},line width=1.2pt]  plot[domain=1.2668226187643103:1.8763973439024366,variable=\t]({1.*16.79559923619067*cos(\t r)+0.*16.79559923619067*sin(\t r)},{0.*16.79559923619067*cos(\t r)+1.*16.79559923619067*sin(\t r)});
\draw [->, shift={(19.37392598872407,9.63923476036045)},line width=1.2pt]  plot[domain=3.3609073531733613:3.970731049757719,variable=\t]({1.*16.777930175868597*cos(\t r)+0.*16.777930175868597*sin(\t r)},{0.*16.777930175868597*cos(\t r)+1.*16.777930175868597*sin(\t r)});
\draw [->, shift={(-8.333612517864772,-6.380920499656055)},line width=1.2pt]  plot[domain=0.2192694750466929:0.8291836207048014,variable=\t]({1.*16.775519638174345*cos(\t r)+0.*16.775519638174345*sin(\t r)},{0.*16.775519638174345*cos(\t r)+1.*16.775519638174345*sin(\t r)});
\draw [->, shift={(14.35091983175103,-6.3818000228558365)},line width=1.2pt]  plot[domain=2.313319799013006:2.9229717803830764,variable=\t]({1.*16.790728012798304*cos(\t r)+0.*16.790728012798304*sin(\t r)},{0.*16.790728012798304*cos(\t r)+1.*16.790728012798304*sin(\t r)});
\draw [->, shift={(-13.389609820737382,9.628583917256625)},line width=1.2pt]  plot[domain=5.454838648429755:6.064638238145913,variable=\t]({1.*16.786790330466452*cos(\t r)+0.*16.786790330466452*sin(\t r)},{0.*16.786790330466452*cos(\t r)+1.*16.786790330466452*sin(\t r)});
\draw [->, shift={(2.9940079090559673,13.282142090759644)},line width=1.2pt]  plot[domain=4.407959240767685:5.017651628291141,variable=\t]({1.*16.794556327595302*cos(\t r)+0.*16.794556327595302*sin(\t r)},{0.*16.794556327595302*cos(\t r)+1.*16.794556327595302*sin(\t r)});
\draw [shift={(2.9803003568273967,6.988860475994255)},line width=1.2pt,->]  plot[domain=-1.5532143950535158:4.014098663792105,variable=\t]({1.*1.*cos(\t r)+0.*1.*sin(\t r)},{0.*1.*cos(\t r)+1.*1.*sin(\t r)}) ;
\draw [shift={(8.89586008462023,-3.249577222650493)},line width=1.2pt,->]    plot[domain=-3.6856496994623047:2.1335046414728542,variable=\t]({1.*1.*cos(\t r)+0.*1.*sin(\t r)},{0.*1.*cos(\t r)+1.*1.*sin(\t r)})  ;
\draw [shift={(-2.911791543957258,-3.2303056150093457)},line width=1.2pt, ->]    plot[domain=0.512248008742023:6.265411724401079,variable=\t]({1.*1.*cos(\t r)+0.*1.*sin(\t r)},{0.*1.*cos(\t r)+1.*1.*sin(\t r)})  ;
\begin{scriptsize}
\draw [fill=blue] (2.9978813827489854,5.9890150341745985) circle (4pt);
\draw[color=blue] (4.1797862649889464,5.901649323548554) node {${\bf 2}$};
\draw [fill=blue] (-2.0401466673309936,-2.7401676660080625) circle (4pt);
\draw[color=blue] (-2.4903946410909225,-1.745594634341063) node {${\bf 3}$};
\draw [fill=blue] (8.040244329607463,-2.7319657083137923) circle (4pt);
\draw[color=blue] (8.472674006886195,-1.6972001145870998) node {${\bf 1}$};

\draw[color=black] (-1,2.101649323548554) node {$1$};
\draw[color=black] (3,-4) node {$\w^2$};
\draw[color=black] (10.1,-4.301649323548554) node {$-1$};
\draw[color=black] (-4.1,-4.301649323548554) node {$-1$};
\draw[color=black] (3,8.5) node {$-1$};
\draw[color=black] (3,-1.3) node {$\w$};
\draw[color=black] (2,1.5) node {$1$};
\draw[color=black] (5,1.5) node {$1$};
\draw[color=black] (7,2.101649323548554) node {$1$};
\end{scriptsize}
\end{tikzpicture} \end{tabular}}
\end{center}
\caption{The decomposition of the graph $G$ on 12 vertices in Figure \ref{ex:12} using the basic  automorphism $\psi=(4,7, 10), (5, 8, 11), (6, 9, 12)$.  Here vertices are labeled blue and edge weights are labeled black, which will be our convention throughout the paper.}\label{threegraphs}
\end{figure}

While there is an automorphism of order three which acts on the first matrix of the decomposed matrix shown in Equation \eqref{ex25}, there is no automorphism permuting all of the vertices of $G$ as $\phi$ does.  Thus, if we continue the recursive process in this example, we fail to account for part of the symmetry found in $\phi$, so that we fail to equitably decompose the graph $G$ in Figure \ref{ex:12}.   This decomposition was done using the transversal $\cT_0=\{4,8,12\}$ (See \cite{BFW, FSSW} for details). One may wonder if a different choice of semi-transversal could give better results. However, simple computations demonstrate that any choice of semi-transversal yields a similarly unsatisfying conclusion. Thus we conclude that the methods contained in \cite{BFW} and \cite{FSSW} cannot be used to \emph{completely} decompose examples like this.
\end{example}

\section{Equitable Decompositions over Prime-Power Automorphisms}\label{sec:pp}
In this section we give a step-by-step method for decomposing a graph over any of its automorphisms $\phi$ of order $p^N$ for some prime $p$ and $N\geq1$, which we refer to as a \emph{prime-power automorphism}. (We note that if $N=1$ then $\phi$ is a basic automorphism.) This result will allow us in the following section to describe the general case of an equitable decomposition of a graph over any of its automorphisms.

To show how a graph can be equitably decomposed over a prime-power automorphism we require the following lemma. 


\begin{lem}\label{lem:Dallas}
For a prime $p$ and $N\geq 2$ let $M$ be the $(f + rp^N) \times (f + rp^N)$ block matrix
\begin{equation}\label{eq:defineM}
M = \left[\begin{array}{llllll}
F & H & H & H & \cdots & H \\
L &  \\
L &  \\
L &  &  & C \\
\vdots &  \\
L &  \\
\end{array}\right],
\end{equation}
where $F$ is an $f \times f$ matrix, $H$ is an $f \times rp^{N-1}$ matrix,  $L$ is an $rp^{N-1} \times f$ matrix, and $C$ is an $rp^N \times rp^N$ matrix. Suppose that the matrix $C$ can be partitioned in two ways:
\begin{equation}\label{blockcirculant}
C =  \left[\begin{array}{lllll}
 C_0 & C_1 & C_2 & \cdots & C_{p^N-1} \\
 C_{p^N-1} & C_0 & C_1 & \cdots & C_{p^N-2} \\
 C_{p^N-2} & C_{p^N-1} & C_0 & \cdots & C_{p^N-3} \\
 \vdots & \vdots & \vdots & & \vdots \\
 C_1 & C_2 & C_3 & \cdots & C_0 \\
\end{array}\right] =
\left[\begin{array}{lllll}
 D_0 & D_1 & D_2 & \cdots & D_{p-1} \\
 D_{p-1} & D_0 & D_1 & \cdots & D_{p-2} \\
 D_{p-2} & D_{p-1} & D_0 & \cdots & D_{p-3} \\
 \vdots & \vdots & \vdots & & \vdots \\
 D_1 & D_2 & D_3 & \cdots & D_0 \\
\end{array}\right]
\end{equation}
where each $D_i$ block is of size $rp^{N-1} \times rp^{N-1}$, and each $C_j$ is of size $r \times r$.
  Then there exists an $(f + rp^{N}) \times (f  +rp^N)$ invertible matrix $T$ such that
\begin{equation}\label{eq:sim_decomp}
T^{-1} M T = \tilde{M} \oplus B_1 \oplus B_2 \oplus \cdots \oplus B_{p^N-p^{N-1}},
\end{equation}
with
\begin{equation}\label{eq:mtilde}
\tilde{M}=\left[\begin{array}{rr} F & pH \\ L & B_0 \end{array}\right], \ B_0 = \sum_{m=0}^{p-1} D_m, \ \text{and} \
B_j = \sum_{m=0}^{p^N-1} \omega^{\gamma_j m}C_m \ \ \text{for} \ \ j = 1,2, \ldots, p^N-p^{N-1},
\end{equation}
where  $\gamma_j$ are the elements of $\{1,2,\dots ,p^N-1\}$ which are not multiples of $p$. Consequently,
\[
\sigma(M) = \sigma\left(\tilde{M}\right)
\cup \sigma(B_1) \cup \sigma(B_2) \cup \cdots \cup \sigma(B_{p^N-p^{N-1}}).
\]
\end{lem}

We refer to any matrix which has the form of the first equality in Equation \eqref{blockcirculant} \emph{block-circulant}\index{block-circulant}.  A matrix which has the form of $C$ which is block-circulant for two different sized block partitions is called \emph{double block-circulant}\index{double block-circulant}.

\begin{proof}
Before we begin, let us establish some useful identities  involving roots of unity. Let $\omega = e^{2\pi i /p^N}$, and $\gamma$ and $\gamma'$ be integers that are both relatively prime to $p$ with $0< \gamma < \gamma' < p^N$.  Thus $\omega^{\gamma}$ is a primitive $p^N$-root of unity. Using the fact that $\sum_{k=0}^{p^N-1}(\omega^{\gamma})^{km}=0$ if $p^N$ does not divide $m$ it is clear that for any integer $a$

\begin{equation}\label{eq:rtsofunity}
\sum\limits_{m=0}^{p-1}\omega^{\gamma (m p^{N-1}+a)}=0, \text{ and } \sum\limits_{m=0}^{p^{N}-1}\omega^{m(\gamma-\gamma')}=0.
\end{equation}

  Let $S$ be the $rp^N \times r(p^{N}-p^{N-1})$ block matrix where $I_r$ is the $r \times r$ identity matrix:
\begin{equation}\label{eq:101}
S = \left[\begin{array}{llll}
 I_r & I_r & \ldots & I_r \\
 \omega^{\gamma_1} I_r & \omega^{\gamma_2} I_r & \ldots & \omega^{\gamma_{q}} I_r \\
 \omega^{2\gamma_1} I_r & \omega^{2\gamma_2} I_r & \ldots & \omega^{2\gamma_{q}} I_r \\
\vdots  & \vdots & & \vdots \\
 \omega^{(p^N-1)\gamma_1} I_r & \omega^{(p^N-1)\gamma_2} I_r & \ldots & \omega^{(p^N-1)\gamma_{q}} I_r
 \end{array}\right],
\end{equation}
where $\omega = e^{2\pi i /p^N}$, and $\omega^{\gamma_1}, \omega^{\gamma_2}, \ldots, \omega^{\gamma_{q}}$ are  the generators of the cyclic group of the $p^N$-roots of unity and $q=p^N-p^{N-1}$. We set


\begin{equation}\label{eq:102}
T = \left[\begin{array}{c | c | c}
I_f & 0 & 0 \\\hline
0 & R^T & S \\
\end{array}\right]
\end{equation}
where $R=\left[\begin{array}{l l l l}I_{rp^{N-1}}& I_{rp^{N-1}}&\cdots& I_{rp^{N-1}}\end{array}\right].$ We consider the product
\[
 \left[\begin{array}{cc  }
I_f & 0 \\[1mm]
0 & \frac{1}{p}R \\[1mm]
 0 & \frac{1}{p^N}S^*   \\
\end{array}\right]T
=\left[\begin{array}{ccc}
I_f & 0 & 0 \\[1mm]
 0 & \frac{1}{p}RR^T
 & \frac{1}{p}RS \\[1mm]
0 & \frac{1}{p^N}S^*R^T & \frac{1}{p^N}S^*S \\

\end{array}\right],
\]
where $S^*$ is the conjugate transpose of $S$. The matrix $RS$ is a $p^{N-1}\times p^{N}$ $r$-block matrix where the $(a,b)^{th}$ block is given by
\[
 I_r\sum\limits_{m=0}^{p-1}\omega^{\gamma_{b}(mp^{N-1}+a-1)}=0, 
 \]
using Equation \eqref{eq:rtsofunity}. Therefore $RS=0$ and similarly, $S^*R^T=0$.
Note that $S^*S$ as a block matrix with $r\times r$ blocks where the $ (a,b)^{th}$ block has the form
\[
I_r\sum\limits_{m=0}^{p^{N}-1}\omega^{m(\gamma_b-\gamma_a)} 
=\left\{\begin{array}{cl}
p^NI_r & \text{if} \ a=b\\
0_r & \text{if} \ a \neq b
\end{array}\right. .
\]
{So, $\frac{1}{p^N}S^*S=I_{r(p^N-p^{N-1})}$}, and therefore
\[
T^{-1}= \left[\begin{array}{cc  }
I_f & 0 \\[1mm]
0 & \frac{1}{p}R \\[1mm]
 0 & \frac{1}{p^N}S^*   \\
\end{array}\right].
\]

Next we show that performing a similarity transformation using $T$ gives the equitable decomposition in Equation \eqref{eq:sim_decomp}.
Let $M$ be the matrix given in Equation \eqref{eq:defineM}.  Letting $P = \left[ H \quad H \quad \cdots \quad H\right]$ and
$Q = [ L^T \quad L^T \quad \cdots \quad L^T]^T$, we have
\[
T^{-1}MT = \left[\begin{array}{cc  }
I_f & 0 \\[1mm]
0 & \frac{1}{p}R \\[1mm]
 0 & \frac{1}{p^N}S^*   \\
\end{array}\right]
\left[\begin{array}{rr} F & P \\ Q & C \end{array}\right]
 \left[\begin{array}{ccc}
I_f & 0 & 0 \\
 0& R^T & S \\
\end{array}\right]
=\left[\begin{array}{ccc}
F & PR^T & PS \\[1mm]
\frac{1}{p}RQ  & \frac{1}{p}RCR^T & \frac{1}{p}RCS\\[1mm]
 \frac{1}{p^N}S^*Q & \frac{1}{p^N} S^*CR^T & \frac{1}{p^N}S^*CS \end{array}\right].\label{blocks}
\]
It is straightforward to verify that
\[
PR^T=pH, \quad RQ=pL, \quad \text{and} \quad RCR^T=p\sum_{m=0}^{p-1}{D_m}=pB_0.
\]
To show that $PS = 0$, we break $P$ into $f \times r$ blocks $H_i$ and observe that $H_i = H_{i + p^{N-1}}$ for $0\leq i\leq p^N-p^{N-1}-1$. The $k$th $f \times r$ block in the product $PS$ is given by
\[ \sum_{i=0}^{p^{N-1}-1}H_i\sum_{m=0}^{p-1}\omega^{\gamma_k (mp^{N-1}+i)} =0,
\]
as in Equation \eqref{eq:rtsofunity}. A similar calculation shows that $S^*Q=0$. Next we consider the matrix product $RCS$.
The $ (a,b)^{th} $ block has the form
\[
 {\sum\limits_{m = 0}^{{p^N} - 1} {C_m}{\omega^{(m+a-1)\gamma_b}\sum\limits_{j=0}^{p-1}{\omega ^{\gamma _b(jp^{N-1})}}} }
 \]
which is the zero matrix by Equation \eqref{eq:rtsofunity} (with $a=0$). Thus $RCS=0$ and similarly $S^*CR^T=0$.

Next we consider $S^*CS$.
The $ (a,b)^{th} $ $r\times r$ block of $ S^*CS$ is given by
\[
 {\sum\limits_{n = 0}^{{p^N} - 1} {\sum\limits_{m = 0}^{{p^N} - 1} {{C_m}{\omega ^{(m + n){\gamma _a}}}{\omega ^{ - n{\gamma _b}}}} } } 
 =\left\{ \begin{array}{cl}
 {{p^N}\sum\limits_{m = 0}^{{p^N} - 1} {{C_m}{\omega ^{m{\gamma _a}}}} }=p^NB_a& \text{ if }a=  b\\
 {\sum\limits_{m = 0}^{{p^N} - 1} {{C_m}{\omega ^{m{\gamma _a}}}\sum\limits_{n = 0}^{{p^N} - 1} {{\omega ^{n({\gamma _a} - {\gamma _b})}}} } }=0 & \text{ if } a \neq b,
 \end{array} \right.
 \]
using Equation \eqref{eq:rtsofunity} for entries where $a\neq b$. Thus,
\[
{S^*}CS = B_1 \oplus B_2 \oplus \ldots \oplus B_{p^N-p^{N-1}}.
\]
Finally, we have
\[
T^{-1}MT =
\left[\begin{array}{cc}
F & pH \\L & B_0
\end{array}\right] \oplus B_1\oplus B_2 \oplus  \ldots \oplus B_{p^N-p^{N-1}}.
\]

\end{proof}

Next we show that if a graph has a prime-power automorphism $\phi$ then any automorphism compatible matrix $M$ of the graph has the form given in Equation \eqref{eq:defineM} if we choose the transversal of the automorphism correctly.

\begin{prop}\label{prop:primepower}
Let $G$ be a graph with automorphism $\phi$ of order $p^N$ for some prime $p$ and integer $N>0$. Let $\cT_0$ be a transversal of the orbits of length $p^N$ of $\phi$, and let
\[
\tcT_0= \bigcup_{m = 0}^{p^{N-1}-1} \cT_m,
\]
 which is a transversal of the orbits of $\phi^p$ when restricted to only vertices contained in orbits of maximal size.  Let $M$ be an automorphism compatible matrix on $G$ and set
\[\begin{array}{cc}
\cT_F = \{ v \in V(G) \ | \ |\so_\phi(v)|< p^N\}, \  f = |\cT_F|,\ F = M[\cT_F, \cT_F],\\[2mm]
  \ H = M[\cT_F, \tcT_0], \ L = [\tcT_0,\cT_F],\ C_m = M[\cT_0, \cT_m], \ and \ D_s = M[\tcT_0, \tcT_s].\end{array}
\]
Then there is a permutation similarity transformation of $M$ which satisfies the conditions of Lemma \ref{lem:Dallas}.
\end{prop}

\begin{proof}
Let $r$ be the number of orbits of length $p^N$, thus $|\cT_k|=r$, and $|\tcT_k|=rp^{N-1}$.
Permute the rows and columns of $M$ so that they are labeled in the order $\cT_F,\cT_0,\cT_1,...,\cT_{p^N-1}$. Abusing notation, we will call this reordered matrix $M$ and let $C$ be the principal submatrix consisting of the last $rp^N$ rows and columns of M. Also let
\begin{equation}\label{deftC}
\cT_C = \cT_0 \cup \cT_1 \cup \ldots \cT_{p^{N}-1}.
\end{equation}
 Notice that  $\phi|_{\cT_C}: \cT_C \to \cT_C$. Hence, $C$ is compatible with $\phi|_{\cT_C}$, since
\[
M[\cT_s,\cT_t]=M[\phi(\cT_s),\phi(\cT_t)]=
\left\{ \begin{array}{ll}
M[\cT_{s+1},\cT_{t+1}] & \text{ if } s, t \neq p^{N}-1\\
M[\cT_{0},\cT_{t+1}] & \text{ if } s = p^{N-1}, t \neq p^{N}-1\\
M[\cT_{s+1},\cT_{0}] & \text{ if } s\neq p^{N-1},  t = p^{N}-1\\
M[\cT_{0},\cT_{0}] & \text{ if } s, t = p^{N}-1\end{array}\right.
\]  Thus $C$ is a block-circulant matrix made up of $r\times r$ blocks.

Since $C$ is automorphism compatible with $\phi|_{\cT_C}$, $C$ must also be automorphism compatible with $\phi^{p^{N-1}}|_{\cT_C}$. Thus
\[
M[\tcT_s,\tcT_t]=M[\phi^{p^{N-1}}(\tcT_s),\phi^{p^{N-1}}(\tcT_t)]=
\left\{ \begin{array}{ll}
M[\tcT_{s+1},\tcT_{t+1}] & \text{ if } s, t \neq p-1\\
M[\tcT_{0},\tcT_{t+1}] & \text{ if } s = p-1, t \neq p-1\\
M[\tcT_{s+1},\tcT_{0}] & \text{ if } s \neq p-1, t = p-1\\
M[\tcT_{0},\tcT_{0}] & \text{ if } s, t = p-1\\
\end{array}\right.
\]
implying $C$ is also block circulant with $rp^{N-1}\times rp^{N-1}$ blocks.

Notice that $\phi^{p^{N-1}}$ fixes $\cT_f$, so that
\[\begin{array}{l}
H=M[\cT_F,\tcT_0]=M[\phi^{mp^{N-1}}(\cT_F),\phi^{mp^{N-1}}(\tcT_0)]= M[\cT_F,\tcT_m], \ and \\[1mm]
L=M[\tcT_0,\cT_F]=M[\phi^{mp^{N-1}}(\tcT_0),\phi^{mp^{N-1}}(\cT_F)]= M[\tcT_m,\cT_F].
\end{array}
\]
Thus, $M$ satisfies the hypotheses of  Lemma \ref{lem:Dallas}.
\end{proof}

Given a graph $G$ with a prime-powered automorphism $\phi$ our goal is to equitably decompose this graph, or equivalently the associated automorphism compatible matrix $M$, by sequentially decomposing $M$ into smaller and smaller matrices. The way we do this is to first use Lemma \ref{lem:Dallas} and Proposition \ref{prop:primepower} to decompose $M$ into the product
\[
\tilde{M} \oplus B_1 \oplus B_2 \oplus \cdots \oplus B_{p^N-p^{N-1}}.
\]
 By virtue of the way in which this decomposition is carried out the smaller matrix $\tilde{M}$ also has a ``smaller" automorphism $\psi$ that can similarly be used to decompose the matrix $\tilde{M}$, which we demonstrate in the following proposition.

\begin{prop}\label{prop:primepowerautom}
Assume the graph $G$, the matrix $M$, and the automorphism $\phi$ satisfy the hypotheses in Proposition \ref{prop:primepower}. Then there exists an  automorphism $\psi\in\Aut(\tilde{M})$ of order $ p^{N-1}$ where
\[
\tM = \left[\begin{array}{rr} F & pH \\ L & B_0 \end{array}\right]
\]
is the matrix described in Lemma \ref{lem:Dallas}.
\end{prop}

\begin{proof}
Let the matrix $M$ be the automorphism compatible matrix with the reordering:
$\cT_F, \cT_0, \cT_1, \ldots, \cT_{p^{N-1}}$.  Recall that under this vertex ordering, each orbit of maximal length of $\phi$ looks like
\[
\so_\phi(i) = (i, i + r, i + 2r, \ldots, i + rp^N).
\]
We define a map $\psi$ on $\cT_F \cup \tcT_0$ (where $\tcT_0$ is defined in Proposition \ref{prop:primepower}) by
\begin{equation}\label{eq:prop33}
\psi(i) = \begin{cases}
\phi(i) & i \notin \cT_{p^{N-1}-1}\\
\phi^{1-p^{N-1}}(i) & otherwise.
\end{cases}
\end{equation}
It is straightforward to verify that $\phi(\cT_F) = \cT_F$, $\phi^k(\cT_m) = \cT_{k+m}$, so that ${\psi: \cT_F \cup \tcT_0 \to \cT_F \cup \tcT_0}$.
%
%
%
%
%
We wish to show for all $i,j\in\cT_F \cup \tcT_0$  $\tM(\psi(i), \psi(j)) = \tM(i,j)$. To do this we first consider the  case where $i$ or $j$ is in $\cT_F$. For $\epsilon_1,\epsilon_2\in\{0,1\}$, ${\tM}(\psi(i), \psi(j))$ is given below.
\[ \begin{array}{lr}
 M(\phi(i),\phi(j)) = M(i,j) = \tM(i,j) &  i \in \cT_F \ j \in \cT_{F} \\
{\tM}(\phi(i), \phi^{1-\epsilon_1 p^{N-1}}(j)) ={\tM}(\phi(i), \phi(j)) = pM(\phi(i),\phi(j)) = pM(i,j) = \tM(i,j) & i \in \cT_F \ j \notin \cT_{F} \\
{\tM}(\phi^{1-\epsilon_2 p^{N-1}}(i), \phi(j)) ={\tM}(\phi(i), \phi(j)) = M(\phi(i),\phi(j)) = M(i,j) = \tM(i,j) & i \notin \cT_F \ j \in \cT_{F} \\
\end{array}
\]
The first equalities in the second and third cases are valid because of the block circulant nature of $M$, as described in Proposition \ref{prop:primepower} and Lemma \ref{lem:Dallas}.

Finally, we consider the case where neither $i$ nor $j$ is in $\cT_F$. Then for $\epsilon_1,\epsilon_2\in\{0,1\}$,
\[\begin{array}{c}
\displaystyle{\tM(\psi(i), \psi(j)) = \tM(\phi^{1 -\epsilon_1 p^{N-1}}(i), \phi^{1 -\epsilon_2 p^{N-1}}(j))= \sum_{m = 0}^{p-1} M(\phi^{1 -\epsilon_1 p^{N-1}}(i), \phi^{1+ (m-\epsilon_2) p^{N-1}}(j)) }\\[1mm]
\displaystyle{=\sum_{m = 0}^{p-1} M(i, \phi^{(m+\epsilon_1-\epsilon_2) p^{N-1}}(j)) = \sum_{m = 0}^{p-1} M\left(i, \left(\phi^{p^{N-1}}\right)^m(j)\right)  = \tM(i, j),
}\\[1mm]
\end{array}
\]
where the second to last equality is true because the sum passes through all $p$ distinct powers of $\phi^{N-1}$ exactly once, and the addition of $\epsilon_1$ and $\epsilon_2$ only changes the order in which this happens.
\end{proof}

Each time we use Propositions \ref{prop:primepower} and \ref{prop:primepowerautom} on a matrix $M$ with automorphism $\phi$ of order $|\phi|=p^N$ we obtain a smaller matrix $\tilde{M}$ with automorphism $\psi$ of order $|\psi|=p^{N-1}$. It is in fact possible to sequentially repeat this process until we ``run out" of powers of $p$. The result is the equitable decomposition of the graph $G$ over $\phi$.


\begin{thm}\label{thm:fullprimedecomp}\textbf{(Equitable Decompositions over Prime-Powered Automorphisms)}
Suppose $G$ is a graph with automorphism $\phi$ where $|\phi|=p^N$ for some prime $p$ and $N\geq 1$. If $M$ is an automorphism compatible matrix of $G$ then by repeated application of Propositions \ref{prop:primepower} and \ref{prop:primepowerautom}, we obtain the equitable decomposition
\[
M_\phi
\oplus\hM_1
 \oplus \hM_2 \oplus \cdots
\oplus\hM_N,
\]
where $M_\phi$ is the divisor matrix associated with $\phi$ and $\hM_i = \left(_iB_1 \oplus\ _iB_2 \oplus \cdots \oplus\ _iB_{p^{N-i+1}-p^{N-i}}\right)$ where $_iB_j$ has size $r_i\times r_i$, where $r_i$ is the number of orbits of $\phi$ with length greater than or equal to $p^{N-i+1}$ and $1\leq j \leq p^{N-i+1}-p^{N-i}$
\end{thm}

\begin{proof}
 By Proposition \ref{prop:primepower} and Lemma \ref{lem:Dallas} we can decompose $M$ into $M_1=\tM_1\oplus\hM_1$, where $\tM_1=\tM$ and $\hM_1= {_1B_1}\oplus{ _1B_2}\oplus \dots\oplus {_1B_{p^N-p^{N-1}}}$.  Also let $\psi$ be the automorphism of $\tM_1$ as in Proposition \ref{prop:primepowerautom}.

Next we will use $\psi$ to decompose $\tM_1$. We pick a transversal $\cT_0$ of the orbits  \textit{of maximal length} of $\psi$, and $\cT_F$ which will contain all the indices belonging to orbits of length less than $p^{N-1}$. If we perform a permutation similarity transformation on $M_1$ so that our indices now appear in the order $\cT_F, \cT_0, \ldots, \cT_{p^{N-1}}$, we can use Proposition \ref{prop:primepower}, and Lemma \ref{lem:Dallas} to complete another decomposition.  We repeat this process $N$ times.

We need to show that   the block matrix $\tM_N$
appearing in the upper left portion of the final decomposition satisfies $\tM_N = M_\phi$, where $M_\phi$ is the divisor matrix obtained from an equitable partition of the original matrix $M$ using the orbits of $\phi$ as the partition set.
To do so, recall that
\[
M_\phi(i,j) = \sum_{r \in \so_\phi(j)} M_{i,r}.
\]
It is easy to verify that each of the indices of the rows and columns of $\tM_N$ correspond to a distinct orbit of length $p^k$.

Suppose that $j$ is an index appearing in $\tM_N$ and in an orbit of length $p^k$ for some $k>0$.  Then in the first $N-k$ decompositions, $j$ will be placed in the set $\cT_F$. Thus, if $\tM_{\kappa-1}$ is the matrix created at the end of the $(\kappa-1)^{th}$ decomposition and $\tilde{M}$ is the matrix created at the end of the $\kappa^{th}$ decomposition, and $1 \leq \kappa \leq N-k$, then  $\tM_{\kappa}(i,j) = \tM_{\kappa-1}(i,j)=M(i,j)$.


We consider the $(N-k+1)^{th}$ decomposition, which begins with an automorphism of order $p^{k}$. Examining the formulas in Equation \eqref{eq:sim_decomp} and \eqref{eq:mtilde} gives
\[
\tM_{N-k+1}(i,j) =  \sum_{m=0}^{p-1} M(i,\phi^{mp^{k-1}}(j)).
 \]
 The next decomposition will yield
 \[
\tM_{N-k+2}(i,j) =  \sum_{m_2=0}^{p-1} \tM_{N-k+1}(i,\phi^{m_2p^{k-2}}(j))= \sum_{m_2=0}^{p-1}\sum_{m_1=0}^{p-1} M(i,\phi^{m_1p^{k-1} + m_2p^{k-2}}(j)).
 \]
Thus, the entries in $\tM_N$ are determined by
\[
\tM_N= \sum_{m_1, m_2, \ldots, m_k  =  0}^{p-1} M(i,\phi^{m_1p^{k-1} + m_2p^{k-2} + \ldots + m_k}(j)) = \sum_{r \in \so_\phi(j)} M_{i,r}.
\]
\end{proof}

We note here that in the case that $N=1$, this process is exactly the same  process outlined in  \cite{BFW} for equitably decomposing over a basic automorphism.

According to Theorem \ref{thm:fullprimedecomp} it is possible to equitably decompose a graph over any of its prime-power automorphisms. Although this is true it is probably not obvious at this point how this type of decomposition can be carried out. What follows is an algorithm detailing the steps involved in this process.

\begin{center}
\emph{Performing Equitable Decompositions Involving Prime-Power Automorphisms}
\end{center}
\begingroup\raggedright\leftskip=20pt\rightskip=20pt

For a graph $G$ with automorphism compatible matrix $M$ and automorphism $\phi$ with $|\phi|=p^N$, set $M_0  = M$, and $\phi_1 = \phi$. To begin we start with $i = 1$ in \emph{Step a.}\m

\vspace{0.1in}

\noindent\emph{\textbf{Step a:}} Choose $\mathcal{T}_F$ to be all elements of the graph G which are contained in orbits of $\phi_i$ with length less than $p^{N-i+1}$.  Choose $\mathcal{T}_0$ to be a transversal of all orbits of $\phi_i$ with length equal to $p^{N-i+1}$, and let
\[
\tilde{\mathcal{T}}_0=\{\mathcal{T}_0,\phi_i({\mathcal{T}_0}),\phi_i^2(\mathcal{T}_0),\dots,\phi_i^{p^{n-i}}({\mathcal{T}_0})\}
\]
and use this $\mathcal{T}_0$ and $\tilde{\mathcal{T}}_0$ to order the matrix $M_{i-1}$ as in Proposition \ref{prop:primepower}. 

\noindent\emph{\textbf{Step b:}} Form the matrix $T$ as described in Lemma \ref{lem:Dallas}. Perform the equitable decomposition of $M_{i-1}$ via a similarity transformation and define
\[
M_{i}=TM_{i-1}T^{-1}.
\]

\noindent\emph{\textbf{Step c:}} Extending the definition of $\psi$ {found in Equation \eqref{eq:prop33} which is described in the proof of Proposition 3.3}, we define $\phi_{i+1}$ by
\begin{equation}\label{eq:6}
\phi_{i+1}(k) = \begin{cases}
\phi_i(k) & k \in \cT_f\cup\cT_0\cup\cT_1\cup\dots\cup\cT_{p^{N-i}-2}\\
\phi_i^{1-p^{N-i}}(k) &  k \in \cT_{p^{N-i}-1}\\
k & otherwise.
\end{cases}
\end{equation}
If $i<N$, then set $i=i+1$ and return to {\it Step a}, otherwise the decomposition is complete.
\par\endgroup
 \vspace{.5cm}


To demonstrate how this algorithm is applied we return to our previous example,  the graph $G$ shown in Figure \ref{ex:12}, which we are now able to fully decompose.

\begin{example}
Consider the matrix $M$ and graph $G$ shown in Figure \ref{ex:12} previously considered in Example \ref{ex:1}. Here, $\phi=\phi_1$ is the automorphism $\phi_1=(1,2,3)(4,5,6,7,8,9,10,11,12)$, which has order $9=3^2$. Thus we will run through {\it Steps a-c} in our algorithm twice, {i.e. we go through two rounds of this algorithm.}
%
 Note that the matrix $M$ has the form guaranteed by Proposition \ref{prop:primepower}. We set $i=1$ and proceed to {\it Step a}.

\begin{center}
\emph{Round 1}
\end{center}
\emph{Step a}:Here,  $\mathcal{T}_F=\{1,2,3\}$ and choose $\mathcal{T}_0=\{4\}$, which gives us $\tilde{\mathcal{T}}_0=\{4,5,6\}$. Thus the adjacency matrix $M_0 = M$ for $G$ given in Figure \ref{ex:12} is already ordered appropriately. \\
\emph{Step b}:
Following Proposition \ref{prop:primepower}, we can decompose {$M_0$ using a similarity transformation to obtain $M_0 ~M_1= \tM_1 + \hM_1$}, where

\[
\tM_1 = \left[\begin{array}{rrr|rrr}
2&1&1&1&0&0\\
1&2&1&0&1&0\\
1&1&2&0&0&1\\\hline
3&0&0&0&1&1\\
0&3&0&1&0&1\\
0&0&3&1&1&0
\end{array}\right], \text{ and }\hM_1 = [\lambda_1]\oplus[\lambda_2]\oplus[\lambda_3]\oplus[\lambda_4]\oplus[\lambda_5]\oplus[\lambda_6]
\]
where (with $\omega = e^{\frac{2 \pi i }{9}}$) Equation \eqref{eq:mtilde} gives
\[
\begin{array}{l}
\lambda_1 = \omega + \omega^3 + \omega^6 +\omega^8 \approx-0.652, \\
\lambda_2= \omega^2 + \omega^6 + \omega^{12} +\omega^{16} \approx-2.879, \\
\lambda_3= \omega^4 + \omega^{12} + \omega^{24} +\omega^{32}\approx0.532, \\
\lambda_4= \omega^5 + \omega^{15} + \omega^{30} +\omega^{40}\approx0.532, \\
\lambda_5= \omega^7 + \omega^{21} + \omega^{42} +\omega^{56}\approx-2.879,\\
\lambda_6= \omega^8 + \omega^{24} + \omega^{48} +\omega^{64}\approx-0.652,
\end{array}
\]
 which are six eigenvalues of the adjacency matrix $M$. The associated {adjacency graph},  the graph whose adjacency matrix is $\tilde{M}_1\oplus\hat{M}_1$, is shown in Figure \ref{ex:21}.\\

\begin{figure}[h!]
\begin{center}
\resizebox{.6\textwidth}{!}{
\begin{tikzpicture}[line cap=round,line join=round,>=triangle 45,x=1.0cm,y=1.0cm, scale = .4]
\draw [line width=1.2pt] (3.,2.)-- (1.42,-0.74);
\draw [line width=1.2pt] (1.42,-0.74)-- (4.582909606369362,-0.738320137979414);
\draw [line width=1.2pt] (4.582909606369362,-0.738320137979414)-- (3.,2.);
\draw [->, shift={(8.746530801304676,3.9975601362061384)},line width=1.2pt]  plot[domain=2.808109490764957:3.47613804206866,variable=\t]({1.*6.083819749803839*cos(\t r)+0.*6.083819749803839*sin(\t r)},{0.*6.083819749803839*cos(\t r)+1.*6.083819749803839*sin(\t r)});
\draw [->,shift={(-2.7619038286835558,3.9914478583805644)},line width=1.2pt]  plot[domain=-0.3327703910467745:0.33383261670080516,variable=\t]({1.*6.096343191097985*cos(\t r)+0.*6.096343191097985*sin(\t r)},{0.*6.096343191097985*cos(\t r)+1.*6.096343191097985*sin(\t r)});
\draw [->, shift={(2.574873153942068,-6.730834430650846)},line width=1.2pt]  plot[domain=1.7612336156607167:2.4286188147797056,variable=\t]({1.*6.1011334338111105*cos(\t r)+0.*6.1011334338111105*sin(\t r)},{0.*6.1011334338111105*cos(\t r)+1.*6.1011334338111105*sin(\t r)});
\draw [->,shift={(-3.190222966731918,3.2423685501799246)},line width=1.2pt]  plot[domain=4.902310947532458:5.570726790087551,variable=\t]({1.*6.09207805863029*cos(\t r)+0.*6.09207805863029*sin(\t r)},{0.*6.09207805863029*cos(\t r)+1.*6.09207805863029*sin(\t r)});
\draw [->,shift={(3.4409228960059806,-6.713365807470568)},line width=1.2pt]  plot[domain=0.7135095119397563:1.3819478161074714,variable=\t]({1.*6.083198533596579*cos(\t r)+0.*6.083198533596579*sin(\t r)},{0.*6.083198533596579*cos(\t r)+1.*6.083198533596579*sin(\t r)});
\draw [->,shift={(9.177088828763083,3.23416245566517)},line width=1.2pt]  plot[domain=3.854546186837524:4.524096448389289,variable=\t]({1.*6.073475173595705*cos(\t r)+0.*6.073475173595705*sin(\t r)},{0.*6.073475173595705*cos(\t r)+1.*6.073475173595705*sin(\t r)});
\draw [->,shift={(2.980
3003568273967,6.988860475994255)},line width=1.2pt]  
  plot[domain=-1.5532143950535158:4.314098663792105,variable=\t]({1.*1.*cos(\t r)+0.*1.*sin(\t r)},{0.*1.*cos(\t r)+1.*1.*sin(\t r)}) ;
\draw [->,shift={(8.89586008462023,-3.249577222650493)},line width=1.2pt]    plot[domain=-3.6856496994623047:2.1335046414728542,variable=\t]({1.*1.*cos(\t r)+0.*1.*sin(\t r)},{0.*1.*cos(\t r)+1.*1.*sin(\t r)})  ;
\draw [->,shift={(-2.911791543957258,-3.2303056150093457)},line width=1.2pt]   plot[domain=0.512248008742023:6.265411724401079,variable=\t]({1.*1.*cos(\t r)+0.*1.*sin(\t r)},{0.*1.*cos(\t r)+1.*1.*sin(\t r)});
\draw [line width=1.2pt] (2.9978813827489854,5.9890150341745985)-- (-2.0401466673309936,-2.7401676660080625);
\draw [line width=1.2pt] (-2.0401466673309936,-2.7401676660080625)-- (8.040244329607463,-2.7319657083137923);
\draw [line width=1.2pt] (8.040244329607463,-2.7319657083137923)-- (2.9978813827489854,5.9890150341745985);
\begin{scriptsize}
\draw [fill=blue] (3.,2.) circle (4.5pt);
\draw[color=blue] (3.3797862649889464,2.1344492655878913) node {${\bf 1}$};
\draw [fill=blue] (1.42,-0.74) circle (4.5pt);
\draw[color=blue] (1.0607506144093768,-0.2652055658528048) node {${\bf 2}$};
\draw [fill=blue] (4.582909606369362,-0.738320137979414) circle (4.5pt);
\draw[color=blue] (4.621528751385895,-0.291008305975823) node {${\bf3}$};
\draw [fill=blue] (2.9978813827489854,5.9890150341745985) circle (4.5pt);
\draw[color=blue] (3.5797862649889464,5.901649323548554) node {${\bf4}$};
\draw [fill=blue] (-2.0401466673309936,-2.7401676660080625) circle (4.5pt);
\draw[color=blue] (-2.1903946410909225,-2.045594634341063) node {${\bf5}$};
\draw [fill=blue] (8.040244329607463,-2.7319657083137923) circle (4.5pt);
\draw[color=blue] (8,-2.0972001145870998) node {${\bf6}$};

\draw[color=black] (5.9797862649889464,2.101649323548554) node {$1$};
\draw[color=black] (0,2.101649323548554) node {$1$};
\draw[color=black] (3,-3.101649323548554) node {$1$};
\draw[color=black] (10.1,-4.301649323548554) node {$2$};
\draw[color=black] (-4.1,-4.301649323548554) node {$2$};
\draw[color=black] (3,8.5) node {$2$};
\draw[color=black] (3.7,4.1) node {$3$};
\draw[color=black] (2.3,4.1) node {$1$};
\draw[color=black] (-.1,-.8) node {$3$};
\draw[color=black] (0.5,-2) node {$1$};
\draw[color=black] (6.1,-.8) node {$3$};
\draw[color=black] (5.5,-2) node {$1$};
\draw[color=black] (2.1,1) node {$1$};
\draw[color=black] (4,1) node {$1$};
\draw[color=black] (3,-1) node {$1$};

\draw [fill=blue] (11,3.) circle (4.5pt);
\draw[color=blue] (11.0797862649889464,2.5344492655878913) node {${\bf 7}$};
\draw (9,4.1) node {$\lambda_1$};
\draw [->, shift={(10,3)},line width=1.2pt]  plot[domain=0:5.47613804206866,variable=\t]({1.*cos(\t r)},{1*sin(\t r)});
\draw [fill=blue] (14,3.) circle (4.5pt);
\draw[color=blue] (14.0797862649889464,2.5344492655878913) node {${\bf 8}$};
\draw (12,4.1) node {$\lambda_2$};
\draw [->, shift={(13,3)},line width=1.2pt]  plot[domain=0:5.47613804206866,variable=\t]({1.*cos(\t r)},{1*sin(\t r)});
\draw [fill=blue] (17,3.) circle (4.5pt);
\draw[color=blue] (17.0797862649889464,2.5344492655878913) node {${\bf 9}$};
\draw (15,4.1) node {$\lambda_3$};
\draw [->, shift={(16,3)},line width=1.2pt]  plot[domain=0:5.47613804206866,variable=\t]({1.*cos(\t r)},{1*sin(\t r)});
\draw [fill=blue] (11,0.) circle (4.5pt);
\draw[color=blue] (11.0797862649889464,-.5344492655878913) node {${\bf 10}$};
\draw (9,1.1) node {$\lambda_4$};
\draw [->, shift={(10,0)},line width=1.2pt]  plot[domain=0:5.47613804206866,variable=\t]({1.*cos(\t r)},{1*sin(\t r)});
\draw [fill=blue] (14,0.) circle (4.5pt);
\draw[color=blue] (14.0797862649889464,-.5344492655878913) node {${\bf 11}$};
\draw (12,1.1) node {$\lambda_5$};
\draw [->, shift={(13,0)},line width=1.2pt]  plot[domain=0:5.47613804206866,variable=\t]({1.*cos(\t r)},{1*sin(\t r)});
\draw [fill=blue] (17,0.) circle (4.5pt);
\draw[color=blue] (17.0797862649889464,-.5344492655878913) node {${\bf 12}$};
\draw (15,1.1) node {$\lambda_6$};
\draw [->, shift={(16,0)},line width=1.2pt]  plot[domain=0:5.47613804206866,variable=\t]({1.*cos(\t r)},{1*sin(\t r)});

\end{scriptsize}
\end{tikzpicture}
}
\end{center}
\caption{The decomposition of the graph $G$ on 12 vertices using  the method outlined in Proposition \ref{prop:primepower} after Round 1.}\label{ex:21}
\end{figure}

\emph{Step c}: Now we form our new automorphism $\phi_2$ on $\tilde{M_2}$. According to Equation \eqref{eq:6} $\phi_2$ acts like $\phi$ on $\cT_F \cup \cT_0 \cup \cT_1 = \{1, 2, 3, 4, 5\}$, while $\phi_2(6) = \phi^{1-3^1}(6) = 4 $. Thus, $\phi_2=(1,2,3)(4,5,6)$. Now, we set $i=2$ and return to {\it Step a}.

\begin{center}
\emph{Round 2}\\
\end{center}

\noindent\emph{Step a}: Here,  
$\cT_f=\{ \}$ and we choose $\cT_0=\tilde \cT = \{1,4\}$.
\\
\emph{Step b}: We reorder the rows and columns of $\tilde{M_1}$ to correspond to our transversals.  The relevant block matrices are

\[
\begin{array}{c}
C_0 = \left[\begin{array}{rr}0 & 3 \\ 1 & 2  \end{array}\right],
C_1 =C_2 =  \left[\begin{array}{rr}1 & 0  \\ 0 & 1  \end{array}\right].
\end{array}
\]
Performing the equitable decomposition as described in Proposition \ref{prop:primepower} gives $M_2 = \tM_2 \oplus \hM_2\oplus \hM_1$, where
%
%
%
\begin{equation}\label{eq:103}
\tM_2=
\left[\begin{matrix}
2&3\\
1&4
\end{matrix}\right], \text{ and }
\hM_2 =
\left[\begin{array}{rr}
-1&3\\
1&1
\end{array}\right]
\oplus
\left[\begin{array}{rr}
-1&3\\
1&1
\end{array}\right] 
\end{equation}
with associated adjacency graph in Figure \ref{ex:22}.

 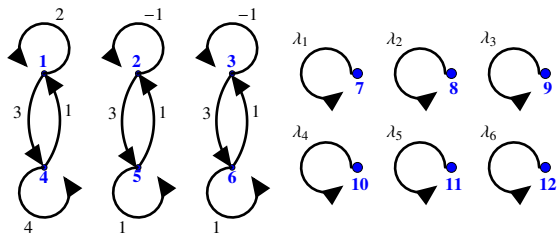
\begin{figure}[h!]
\begin{center}
\resizebox{.5\textwidth}{!}{
\begin{tikzpicture}[line cap=round,line join=round,>=triangle 45,x=1.0cm,y=1.0cm,scale = .5]
\draw [->,shift={(4.000882228420549,1.5022156216448062)},line width=1.2pt]  plot[domain=2.499012508004936:3.7660030032164338,variable=\t]({1.*2.499377469702013*cos(\t r)+0.*2.499377469702013*sin(\t r)},{0.*2.499377469702013*cos(\t r)+1.*2.499377469702013*sin(\t r)});
\draw [->,shift={(0.,1.5026143426529952)},line width=1.2pt]  plot[domain=-0.6443371736649031:0.642663994098626,variable=\t]({1.*2.50156947989587*cos(\t r)+0.*2.50156947989587*sin(\t r)},{0.*2.50156947989587*cos(\t r)+1.*2.50156947989587*sin(\t r)});
\draw [->,shift={(3.0032847116898718,1.5026143426529952)},line width=1.2pt]  plot[domain=-0.6451267165303944:0.6434527688088529,variable=\t]({1.*2.4989441380947057*cos(\t r)+0.*2.4989441380947057*sin(\t r)},{0.*2.4989441380947057*cos(\t r)+1.*2.4989441380947057*sin(\t r)});
\draw [->,shift={(7.035206606035682,1.4574110295747007)},line width=1.2pt]  plot[domain=2.4930216434192602:3.7630423236006214,variable=\t]({1.*2.5537514876998237*cos(\t r)+0.*2.5537514876998237*sin(\t r)},{0.*2.5537514876998237*cos(\t r)+1.*2.5537514876998237*sin(\t r)});
\draw [->,shift={(6.0132585663861855,1.4574110295747007)},line width=1.2pt]  plot[domain=-0.6329017756440738:0.6602101737627059,variable=\t]({1.*2.463978172217413*cos(\t r)+0.*2.463978172217413*sin(\t r)},{0.*2.463978172217413*cos(\t r)+1.*2.463978172217413*sin(\t r)});
\draw [->,shift={(10.014935069490875,1.486678773657992)},line width=1.2pt]  plot[domain=2.497418537268388:3.777259109258474,variable=\t]({1.*2.519941362087911*cos(\t r)+0.*2.519941362087911*sin(\t r)},{0.*2.519941362087911*cos(\t r)+1.*2.519941362087911*sin(\t r)});
\begin{scriptsize}
\draw [fill=blue] (8.,3.) circle (2.5pt);
\draw[color=blue] (7.975102560285598,3.3921983395888806) node {$\bf 3$};
\draw (8.5,4.9) node {$-1$};
\draw [->, shift={(8,3.8)},line width=1.2pt]  plot[domain=-1.6:4,variable=\t]({.8*cos(\t r)},{.8*sin(\t r)});

\draw [fill=blue] (8.,0.) circle (2.5pt);
\draw[color=blue] (7.996312802617946,-0.31959406857202155) node {$\bf 6$};
\draw (7.5,-1.9) node {$1$};
\draw [->, shift={(8,-.8)},line width=1.2pt]  plot[domain=1.6:7,variable=\t]({.8*cos(\t r)},{.8*sin(\t r)});

\draw [fill=blue] (5.,3.) circle (2.5pt);
\draw[color=blue] (4.9420379067598335,3.4134085819212285) node {$\bf 2$};
\draw (5.5,4.9) node {$-1$};
\draw [->, shift={(5,3.8)},line width=1.2pt]  plot[domain=-1.6:4,variable=\t]({.8*cos(\t r)},{.8*sin(\t r)});

\draw [fill=blue] (5.,0.) circle (2.5pt);
\draw[color=blue] (4.963248149092181,-0.2771735839073255) node {$\bf 5$};
\draw (4.5,-1.9) node {$1$};
\draw [->, shift={(5,-.8)},line width=1.2pt]  plot[domain=1.6:7,variable=\t]({.8*cos(\t r)},{.8*sin(\t r)});

\draw [fill=blue] (2.,3.) circle (2.5pt);
\draw[color=blue] (1.9726039802311117,3.4134085819212285) node {$\bf 1$};
\draw (2.5,4.9) node {$2$};
\draw [->, shift={(2,3.8)},line width=1.2pt]  plot[domain=-1.6:4,variable=\t]({.8*cos(\t r)},{.8*sin(\t r)});

\draw [fill=blue] (2.,0.) circle (2.5pt);
\draw[color=blue] (1.9938142225634599,-0.31959406857202155) node {$\bf 4$};
\draw (1.5,-1.9) node {$4$};
\draw [->, shift={(2,-.8)},line width=1.2pt]  plot[domain=1.6:7,variable=\t]({.8*cos(\t r)},{.8*sin(\t r)});

\draw[color=black] (1.1514557409312794,1.7590096799980837) node {$3$};
\draw[color=black] (2.7725419771985965,1.8226404069951276) node {$1$};
\draw[color=black] (5.784396388392013,1.7590096799980837) node {$1$};
\draw[color=black] (4.1845203944570445,1.7377994376657355) node {$3$};
\draw[color=black] (8.753830314920734,1.8014301646627797) node {$1$};
\draw[color=black] (7.196374805650462,1.7377994376657355) node {$3$};

\draw [fill=blue] (12,3.) circle (4.5pt);
\draw[color=blue] (12.0797862649889464,2.5344492655878913) node {${\bf 7}$};
\draw (10.2,4.1) node {$\lambda_1$};
\draw [->, shift={(11,3)},line width=1.2pt]  plot[domain=0:5.47613804206866,variable=\t]({.8*cos(\t r)},{.8*sin(\t r)});
\draw [fill=blue] (15,3.) circle (4.5pt);
\draw[color=blue] (15.0797862649889464,2.5344492655878913) node {${\bf 8}$};
\draw (13.2,4.1) node {$\lambda_2$};
\draw [->, shift={(14,3)},line width=1.2pt]  plot[domain=0:5.47613804206866,variable=\t]({.8*cos(\t r)},{.8*sin(\t r)});
\draw [fill=blue] (18,3.) circle (4.5pt);
\draw[color=blue] (18.0797862649889464,2.5344492655878913) node {${\bf 9}$};
\draw (16.2,4.1) node {$\lambda_3$};
\draw [->, shift={(17,3)},line width=1.2pt]  plot[domain=0:5.47613804206866,variable=\t]({.8*cos(\t r)},{.8*sin(\t r)});
\draw [fill=blue] (12,0.) circle (4.5pt);
\draw[color=blue] (12.0797862649889464,-.5344492655878913) node {${\bf 10}$};
\draw (10.2,1.1) node {$\lambda_4$};
\draw [->, shift={(11,0)},line width=1.2pt]  plot[domain=0:5.47613804206866,variable=\t]({.8*cos(\t r)},{.8*sin(\t r)});
\draw [fill=blue] (15,0.) circle (4.5pt);
\draw[color=blue] (15.0797862649889464,-.5344492655878913) node {${\bf 11}$};
\draw (13.2,1.1) node {$\lambda_5$};
\draw [->, shift={(14,0)},line width=1.2pt]  plot[domain=0:5.47613804206866,variable=\t]({.8*cos(\t r)},{.8*sin(\t r)});
\draw [fill=blue] (18,0.) circle (4.5pt);
\draw[color=blue] (18.0797862649889464,-.5344492655878913) node {${\bf 12}$};
\draw (16.2,1.1) node {$\lambda_6$};
\draw [->, shift={(17,0)},line width=1.2pt]  plot[domain=0:5.47613804206866,variable=\t]({.8*cos(\t r)},{.8*sin(\t r)});

\end{scriptsize}
\end{tikzpicture}
}
\end{center}
\caption{The decomposition of the graph $G$ on 12 vertices using  the method outlined in Proposition \ref{prop:primepower} after Round 2.}\label{ex:22}

 \end{figure}

\noindent\emph{Step c}: Because $i=2$, the decomposition is complete and there is no need to find $\phi_2$.\\[2mm]

\noindent It is worth noting that in this equitable decomposition we have recovered the divisor matrix of the equitable partition associated with the prime-power automorphism $\phi$. This is the matrix $\tilde{M}_2=\left[\begin{matrix}2 & 3 \\1&4\end{matrix}\right]$ seen in Equation \eqref{eq:103}.
\end{example}

\section{General Equitable Decompositions}\label{GED}
In general,  the order of an automorphism $\phi$ of a graph $G$ will not be prime-powered but will have order $\phi=p^N \ell$ where {$p$ is a prime which is relatively prime to $\ell$}. In this case neither Theorem \ref{thm:fullprimedecomp} nor any previous result guarantees that it is possible to create an equitable decomposition of $G$ with respect to $\phi$. In this section we show that this can, in fact, be done.

\begin{remark}\label{rem:careful_transversals}
When performing an equitable decomposition of a graph using an automorphism $\phi$ whose order is not prime-powered (say, $|\phi| = p^N \ell$), our strategy will be to create an automorphism $\psi = \phi^\ell$ of order  $p^N$ and follow the procedure set out in Theorem \ref{thm:fullprimedecomp}.  In order to guarantee that the resulting decomposed matrix still has an automorphism of order $\ell$, we must restrict our method of choosing transversals.  This is done using the following rules:
\begin{itemize}
\item If $a \in V(G)$ was chosen to be in $\cT_0$ in a previous round, $a$ must appear in $\cT_0$ in the next round as well.
\item If $a \in V(G)$ is chosen to be in $\cT_0$ in a certain round, and $|\so_\phi(a)| = p^km$, then $\phi^{p^k}(a), \phi^{2p^k}(a), \ldots, \phi^{(m-1)p^k}(a)$ must also be in $\cT_0$.
\end{itemize}
\end{remark}


\begin{prop}\label{prop:dd}
Let $\phi$ be an automorphism of a graph $G$ with automorphism compatible \index{automorphism compatible}matrix $M$.  Suppose that $\phi$ has order $p^N\ell$, with $p$ a prime which does not divide $\ell$. Then $\psi = \phi^{\ell}$ is an automorphism of $G$ of order $p^N$. Moreover, it is possible to construct an automorphism $\tilde{\phi}$ associated with the equitable decomposition of $M$ over $\psi$ of order $\ell$ such that the divisor matrix \index{divisor matrix} $M_\phi=
(M_\psi)_{\tilde{\phi}}$.
\end{prop}


\begin{proof}
Let $M$ be an automorphism compatible matrix of the graph $G$ and $\phi\in Aut(G)$ of order $|\phi|=p^N \ell$ where $p$ is a prime that does not divide $\ell$. Then $\psi=\phi^{\ell}$ has order $p^N$, which allows us to use Theorem \ref{thm:fullprimedecomp} to decompose $M$ with respect to this automorphism.

To carry out this decomposition we follow the procedure outlined in the previous propositions and theorems, choosing our transversals according to the guidelines in Remark \ref{rem:careful_transversals}.

Note that since $p^k$ and $\ell$ are relatively prime, there exist integers $\alpha$ and $\beta$ such that
 $1=\ell \alpha+p^N \beta$.
%
We define our automorphism $\tilde{\phi}: V(G) \to V(G)$ by
\[
\tilde{\phi}(a) =
 \phi^{(1 - \ell \alpha)}(a)  = \phi^{p^k(p^{N-k})\beta}(a).
\]
The second equality above demonstrates that $\tp$ is closed on each $\cT_m$.

We now show that $\tp$ is an automorphism of $\tilde{M}_k$ for each $k = 0, \ldots, N$.
By hypothesis, $M = \tilde{M}_0$ is compatible with $\phi^{1-\ell \alpha}$. We then assume that $\tM_{k-1}$ is as well. Let $\cT_F$, $\cT_0, \ldots, \cT_{p^k-1}$ be defined in Round $k$ of the decomposition. Note that since $\tp$ is closed on each transversal,  $a$ and $\phi^{1-\ell \alpha}(a)$ must both be in $\cT_F$ or neither, and  if $a \in \cT_m$, then
 $
\phi^{1-\ell \alpha}(a) $
 must also be an element of $\cT_m$. Thus, 
 $\phi^{1 - \ell \alpha}$ is a closed map on the indices of $\tM_k$.

We wish to show that for any  $a, b \in \cT_F \cup \tilde \cT_0$,
 \begin{equation}\label{eq:compatible}
 \tM_k(\phi^{1-\ell \alpha}(a),\phi^{1-\ell \alpha}(b)) = \tM_k(a,b).\end{equation}
Lemma \ref{lem:Dallas} and  Proposition \ref{prop:primepower} give
\begin{center}
$
 \tM_k(\phi^{1-\ell \alpha}(a), \phi^{1-\ell \alpha}(b)) = \begin{cases}
 \tM_{k-1}(\phi^{1-\ell \alpha}(a), \phi^{1-\ell \alpha}(b)) & \text{ if } b \in \cT_F \\
 p\tM_{k-1}(\phi^{1-\ell \alpha}(a), \phi^{1-\ell \alpha}(b)) & \text{ if } b \notin \cT_F, \ a \in \cT_F \\
 \sum_{m=0}^{p-1}\tM_{k-1}(\phi^{1-\ell \alpha}(a), \phi^{m\ell p^{k-1}}\phi^{1-\ell \alpha}(b)) & \text{ if } b \notin \cT_F, \ a \notin \cT_F  \\
 \end{cases}
 $\\[2mm]

$
\hspace*{4cm} =\begin{cases}
 \tM_{k-1}(a,b)= \tM_k(a,b) & \text{ if } b \in \cT_F \\
 p\tM_{k-1}(a, b) = \tM_k(a,b) & \text{ if } b \notin \cT_F, \ a \in \cT_F \\
 \sum_{m=0}^{p-1}\tM_{k-1}(a, \phi^{m\ell p^{k-1}}(b)) = \tM_k(a,b)& \text{ if } b \notin \cT_F, \ a \notin \cT_F  \\
 \end{cases}.
$
\end{center}
All three cases give the desired result in Equation \eqref{eq:compatible} because, by assumption, $\tM_{k-1}$ is compatible with $\tp$. 


We  now consider the block matrices $_kB_j$ in $\hM_k$ (cf. Theorem \ref{thm:fullprimedecomp}), 
setting $i = N-k+1$ in the {$k$th decomposition} and $\omega$ to be a primitive $p^i$th root of unity. Note that
\[\begin{array}{c}
\displaystyle{
_kB_j(\tilde{\phi}(a),\tilde{\phi}(b))
= \sum_{m=0}^{p^i-1}\omega^{m\gamma_j}\tM_{k-1}(\tilde{\phi}(a), \phi^{\ell m}(\tilde{\phi}(b))
=\sum_{m=0}^{p^i-1}\omega^{m\gamma_j}\tM_{k-1}(\tilde{\phi}(a), \tp \circ \phi^{\ell m}(b)}\\[2mm]
\displaystyle{
=\sum_{m=0}^{p^i-1}\omega^{m\gamma_j}\tM_{k-1}(a,  \phi^{\ell m}(b)
={_kB_j(a,b)}
}

\end{array}
\]
since $\tM_{k-1}$ is compatible with $\tp$. Thus, $\tp$ is an automorphism on the decomposed matrix $M_k = \tM_k + \hat{M}_k$ for each $k = 0, 1, \ldots, N$.  It is straightforward to verify that $\tp$ has order $\ell$ on $M_N$, since the row and column indices in this matrix contain representatives from each orbit of $\phi$.

Finally, we need to show that by first decomposing $M$ using $\psi$ then using $\tp$ the result is the divisor matrix $M_{\phi}$. To see this we assume that $|\so_\phi (b)| = p^k m$ and note that
\[
\left(M_{\phi^\ell}\right)_{\tp} (a,b)= \sum_{r \in \so_{\tp}(b)}   M_{\phi^\ell}(a,r) =  \sum_{r \in \so_{\tp}(b)}  \sum_{s \in \so_{\phi^\ell}(r)} M(a,s)
\]
Since $m$ is relatively prime to both $\beta$ and $p$, it is clear from the construction of $\alpha$ and $\beta$ that
\[
 \so_{\tp}(b) = \{\phi^{(1-\ell \alpha) t}(b) = \phi^{\beta p^{N-k}p^k  t}(b)\ | \ t  = 0, \ldots, m-1 \}= \so_{\phi^{p^k}}(b).
 \] 
 Thus, 
\[
\bigcup_{r \in \so_{\tp}(b)} \so_{\phi^\ell}(r) = \{ \phi^{\ell x+p^k t}(b) \ | \ t  = 0, \ldots, m-1, \ x = 0, \ldots, p^k \} = \so_{\phi}(b),
\]
and
\[
\left(M_{\phi^\ell}\right)_{\tp} (a,b) = \sum_{r \in \so_{\tp}(b)}  \sum_{s \in \so_{\phi^\ell}(r)} M(a,s)
=\sum_{r \in \so_{\phi}(b)}   M(a,r)  = M_{\phi}(a,b).
\]

\end{proof}

By repeated application of Theorem \ref{thm:fullprimedecomp} and Proposition \ref{prop:dd} we can at this point state the general theorem of equitable decompositions.



\begin{thm}\label{thm:gendecomp}\textbf{(Equitable Decompositions over Arbitrary Automorphisms)}
Let $G$ be a graph, $\phi$ be any automorphism of $G$, and $M$ be an automorphism compatible matrix of $G$. Then there exists an invertible matrix $T$ that can be explicitly constructed such that
\begin{equation}
T^{-1}MT=M_{\phi}\oplus B_1\oplus B_2\oplus\cdots B_{{K}}
\end{equation}
where $M_\phi$ is the divisor matrix associated with $\phi$.
Thus
\[\sigma(M) = \sigma\left(M_\phi \right)
\cup \sigma(B_1) \cup \sigma(B_2) \cup \cdots \cup \sigma(B_{{K}}).
\]
\end{thm}

We now give an algorithm for equitably decomposing a graph with respect to any of its automorphisms.

\begin{center}
\emph{Performing Equitable Decompositions of General Type}
\end{center}

Let $G$ be a graph with automorphism compatible matrix $M$ and $\phi$ of order $\ell$ with prime factorization ${\ell=p_0^{N_0}p_1^{N_1} \cdots p_{h-1}^{N_{h-1}}}$. Initially set $M_0  = M$, $\ell_0 = \ell$, and $\phi_0 = \phi$.  We perform $h$ sequential decompositions of $M$, one for each prime in the factorization of $\ell$. To begin we start with $i = 0$, and move to \emph{Step A.}\m

\vspace{0.1in}

\begingroup\raggedright\leftskip=20pt\rightskip=20pt


\noindent\emph{\textbf{Step A:}} \emph{Let $\ell_{i+1} = \ell_i/p_i^{N_i}$.  Form the prime-power automorphism $\psi_{i} = \phi_i^{\ell_{i+1}}$}, which has order $p_i^{N_i}$.\\
\noindent\emph{\textbf{Step B:}} \emph{Perform the $N_i$ equitable decompositions of $M_i$ using the algorithm described in Section \ref{sec:pp}. Throughout this process we choose a semi-transversal as prescribed in Remark  \ref{rem:careful_transversals}.  Finally we define $M_{i+1}$ to be the resulting matrix of the above algorithm.}


\noindent\emph{\textbf{Step C:}} \emph{Define $\phi_{i+1}=\tilde{\phi}_i=\phi_i^{1-\ell_i \alpha}$, where $\alpha$ is the integer chosen so that $1=\ell_{i+1} \alpha+p_i^{N_i}\beta$ as described in the proof of Proposition \ref{prop:dd}.
If $i < h$, then set $i = i+1$ and return to \emph{Step A}.  Otherwise, the decomposition is complete.}


\par\endgroup

\vspace{0.1in}

\noindent The procedure described in Steps $A$--$C$ allows one to sequentially decompose a matrix $M$ over any of its automorphisms. By extension we refer to the resulting matrix as an \emph{equitable decomposition} of $M$ over $\phi$. The following example illustrates an equitable decomposition over an automorphism that is neither basic, separable, nor prime-powered, i.e. an equitable partition that cannot be done by any previously given algorithm.

\begin{example}\label{ex:18}
Consider the graph $G$ shown in Figure \ref{fig:ex31}. Here we consider its adjacency matrix $M$.

\begin{figure}[h!]
\begin{center}
\includegraphics[scale=.45]{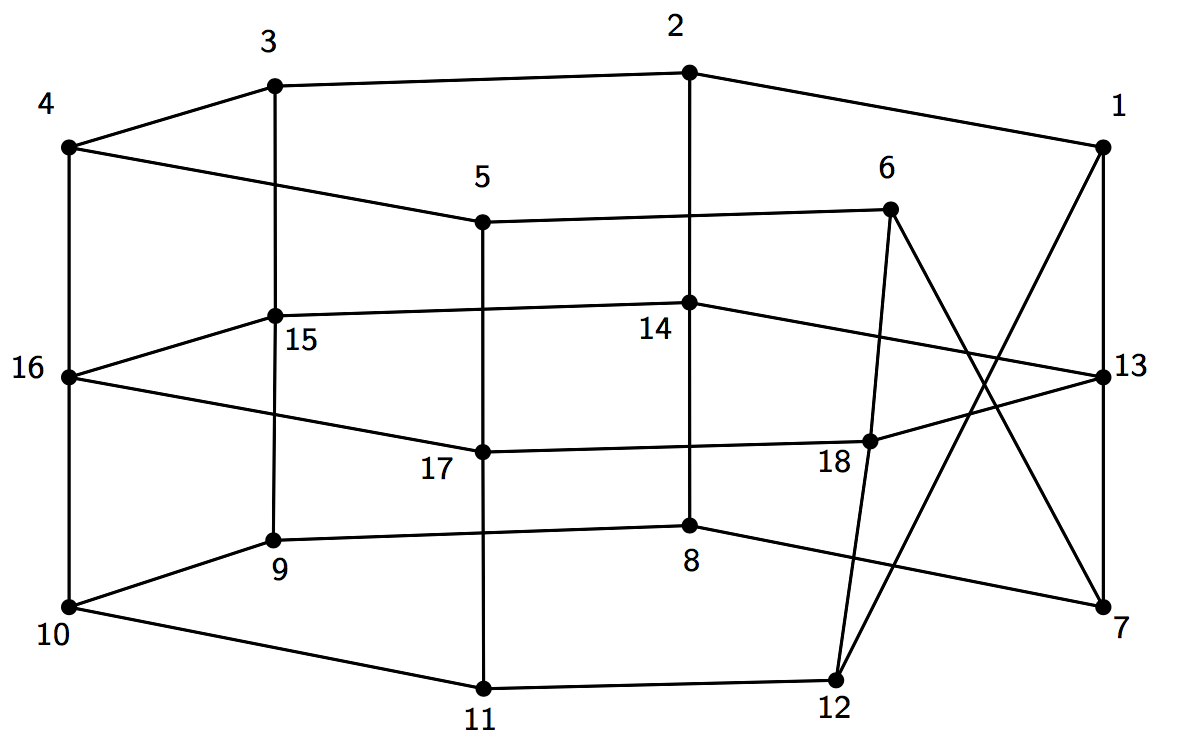}
\end{center}
\caption{The graph $G$ with automorphism of order 12. }\label{fig:ex31}
\end{figure}

This graph has the automorphism
$$\phi=(1,2,3,4,5,6,7,8,9,10,11,12)(13,14,15,16,17,18),$$
which has order $12= 2^2 \cdot 3$. Thus in the prime decomposition of $|\phi|$, $p_0=2,N_0=2$ and $p_1=3,N_1=1$.  Because there are two distinct prime factors in this prime decomposition we will go through the Steps A-C in our algorithm twice in two rounds in order to fully decompose the graph $G$. We start with $i=0$, $M_0=M$, $\ell_0=12$, and $\phi_0=\phi$.

\begin{center}
\emph{Round I}
\end{center}

\noindent \emph{Step A}:  We start with $\ell_1=\ell_0/p_0^{N_0}=12/2^2=3$ so that
\[
\psi_0=\phi_0^3=(1,4,7,10)(2,5,8,11)(3,6,9,12)(13,16)(14,17)(15,18).
\]
\emph{Step B}: Now we run though the algorithm given in Section \ref{sec:pp} for decomposing a graph over a prime-power automorphism. In this case we will require two rounds for the automorphism $\psi_0$.

In Round 1 of this sub-decomposition, we have three orbits of maximal length.  When choosing our semi-transversal $\cT_0$, we are free to choose any element from the first orbit, but based on that choice the other two elements of $\cT_0$ are determined by the rules in Remark \ref{rem:careful_transversals}. We choose vertex 1 to be in $\cT_0$.  Therefore $\cT_0=\{1,\phi_0^4(1),\phi_0^{2\cdot 4}(1)\}=\{1,5,9\}$. Thus,
\[
\cT_0=\{1,5,9\}, \ \cT_1=\{4,8,12\},\ \cT_2=\{7,11,3\}, \text{ and }\cT_3=\{10,2,6\}.
\]
 The remaining vertices are put into $\cT_F=\{13,14,15,16,17,18\}$ since they are contained in orbits of $\psi_{0}$ whose order is not maximal.  The relevant block matrices for this stage of the decomposition are
 \[
 \begin{array}{c}
 F = \left[\begin{array}{rrrrrr}
 0 & 1 & 0 & 0 & 0 & 1\\
 1 & 0 & 1 & 0 & 0 & 0 \\
 0 & 1 & 0 & 1 & 0  & 0 \\
 0 & 0 & 1 & 0 & 1 & 0 \\
 0 & 0 & 0 & 1 & 0 & 1 \\
 1 & 0 & 0 & 0 & 1 & 0 \\
 \end{array}\right], \
 H = L^T =  \left[\begin{array}{rrrrrr}
 1 & 0 & 0 & 0 & 0 & 0\\
 0 & 0 & 0 & 0 & 1 & 0 \\
 0 & 0 & 1 & 0 & 0 & 0 \\
 0 & 0 & 0 & 1 & 0 & 0\\
 0 & 1 & 0 & 0 & 0 & 0 \\
 0 & 0 & 0 & 0 & 0  & 1 \\
 \end{array}\right],\\[4mm]
 C_0 =C_2 =  \left[\begin{array}{rrr}
 0 & 0 & 0 \\
 0 & 0 & 0 \\
 0 & 0 & 0 \\
 \end{array}\right], \
 C_1 = \left[\begin{array}{rrr}
 0 & 0 & 1 \\
 1 & 0 & 0 \\
 0 & 1 & 0 \\
 \end{array}\right], \
 C_3 = \left[\begin{array}{rrr}
 0 & 1 & 0 \\
 0 & 0 & 1 \\
 1 & 0 & 0 \\
 \end{array}\right],\\[6mm]
  D_0 =  \left[\begin{array}{c|c}
 C_0 & C_1\\\hline
 C_3 & C_0
 \end{array}\right], \
 D_1 =  \left[\begin{array}{c|c}
C_2 & C_3 \\\hline
C_1 & C_2
 \end{array}\right]
 \end{array}
 \].
The first round of this algorithm results in a decomposed matrix $M_1 = \tM_1 \oplus \hM_1$ where
\[
\tM_1 = \left[
 \begin{array}{c|c}
 F  & 2H \\\hline
 L &  \begin{array}{c|c}
0 & C_1 + C_3\\\hline
 C_1 + C_3 & 0 \\
 \end{array}
 \end{array}\right],
 \text{ and }
 \hM_1 =  \left[\begin{array}{rrr}
  & -i & i\\
 i & 0 & -i \\
 -i & i & 0 \\
 \end{array}\right]\oplus \left[\begin{array}{rrr}
  & i & -i\\
 -i & 0 & i \\
 i & -i & 0 \\
 \end{array}\right].
\]
\begin{figure}[h!]
\begin{center}
\resizebox{.8\textwidth}{!}{
\begin{tikzpicture}[line cap=round,line join=round,>=triangle 45,x=1.0cm,y=1.0cm,scale = 1]
\draw [ line width=.8pt] (1.,4.)-- (0.,5.732050807568879);
\draw [line width=.8pt] (0.,5.732050807568879)-- (4,5.73);
\draw [line width=.8pt] (1.,7.464101615137759)-- (3.,7.464101615137758); 
\draw [line width=.8pt] (3.,7.464101615137758)-- (1, 4);
\draw [line width=.8pt] (4.,5.732050807568878)-- (3.,4.);
\draw [line width=.8pt] (3.,4.)-- (1.,7.464);
\draw [<-,shift={(-0.506242399292126,7.755413437039516)},line width=.8pt]  plot[domain=-0.19104433320074232:1.2289418560986745,variable=\t]({1.*1.5341540805945852*cos(\t r)+0.*1.5341540805945852*sin(\t r)},{0.*1.5341540805945852*cos(\t r)+1.*1.5341540805945852*sin(\t r)});
\draw [<-,shift={(1.90855230376666,9.134664015740306)},line width=.8pt]  plot[domain=3.106810939765217:4.214271890312301,variable=\t]({1.*1.901642979895714*cos(\t r)+0.*1.901642979895714*sin(\t r)},{0.*1.901642979895714*cos(\t r)+1.*1.901642979895714*sin(\t r)});
\draw [<-,shift={(-0.988899471914504,4.041427734805917)},line width=.8pt]  plot[domain=1.0415299486135894:2.113232867491565,variable=\t]({1.*1.9586037219691117*cos(\t r)+0.*1.9586037219691117*sin(\t r)},{0.*1.9586037219691117*cos(\t r)+1.*1.9586037219691117*sin(\t r)});
\draw [<-,shift={(-1.0097368927336448,7.205719686540119)},line width=.8pt]  plot[domain=4.124844986010433:5.313103137274307,variable=\t]({1.*1.786412202543338*cos(\t r)+0.*1.786412202543338*sin(\t r)},{0.*1.786412202543338*cos(\t r)+1.*1.786412202543338*sin(\t r)});
\draw [<-,shift={(2.581407759435215,2.0453070123650905)},line width=.8pt]  plot[domain=2.2510199347745874:3.0694809985285096,variable=\t]({1.*2.5142941708382085*cos(\t r)+0.*2.5142941708382085*sin(\t r)},{0.*2.5142941708382085*cos(\t r)+1.*2.5142941708382085*sin(\t r)});
\draw [<-,shift={(-0.8372203359588033,3.830918020890392)},line width=.8pt]  plot[domain=-1.0544572683429871:0.09177289446649778,variable=\t]({1.*1.8449843572562339*cos(\t r)+0.*1.8449843572562339*sin(\t r)},{0.*1.8449843572562339*cos(\t r)+1.*1.8449843572562339*sin(\t r)});
\draw [<-,shift={(5.287994450762985,4.218264940237373)},line width=.8pt]  plot[domain=3.236700592479762:4.137088091502714,variable=\t]({1.*2.298381646041195*cos(\t r)+0.*2.298381646041195*sin(\t r)},{0.*2.298381646041195*cos(\t r)+1.*2.298381646041195*sin(\t r)});
\draw [<-,shift={(2.0079006915267903,2.22836543808719)},line width=.8pt]  plot[domain=0.030288476888756067:1.0603148999141336,variable=\t]({1.*2.030504877816603*cos(\t r)+0.*2.030504877816603*sin(\t r)},{0.*2.030504877816603*cos(\t r)+1.*2.030504877816603*sin(\t r)});
\draw [<-,shift={(5.,7.433917393402623)},line width=.8pt]  plot[domain=4.181144364102449:5.24363359666693,variable=\t]({1.*1.9739173934026226*cos(\t r)+0.*1.9739173934026226*sin(\t r)},{0.*1.9739173934026226*cos(\t r)+1.*1.9739173934026226*sin(\t r)});
\draw [<-,shift={(5.,4.)},line width=.8pt]  plot[domain=1.0471975511965994:2.0943951023931957,variable=\t]({1.*2.*cos(\t r)+0.*2.*sin(\t r)},{0.*2.*cos(\t r)+1.*2.*sin(\t r)});
\draw [<-,shift={(1.8500027892592914,9.28873668048161)},line width=.8pt]  plot[domain=5.274762893612817:6.238985096185884,variable=\t]({1.*2.156800108121701*cos(\t r)+0.*2.156800108121701*sin(\t r)},{0.*2.156800108121701*cos(\t r)+1.*2.156800108121701*sin(\t r)});
\draw [<-,shift={(5.562635015790512,7.131798835625496)},line width=.8pt]  plot[domain=2.217922766812652:3.012639915806461,variable=\t]({1.*2.5840904321302696*cos(\t r)+0.*2.5840904321302696*sin(\t r)},{0.*2.5840904321302696*cos(\t r)+1.*2.5840904321302696*sin(\t r)});
\draw [line width=.8pt] (-1.9999783181369541,5.7188805974618955)-- (0.008059480959714316,9.200793082314224);
\draw [line width=.8pt] (0.008059480959714316,9.200793082314224)-- (4.00469641476435,9.193436698311338);
\draw [line width=.8pt] (4.00469641476435,9.193436698311338)-- (6.,5.7320508075688785);
\draw [line width=.8pt] (6.,5.7320508075688785)-- (4.037474256249794,2.2898569351788662);
\draw [line width=.8pt] (4.037474256249794,2.2898569351788662)-- (0.07364803520534302,2.2264598291842406);
\draw [line width=.8pt] (0.07364803520534302,2.2264598291842406)-- (-1.9999783181369541,5.7188805974618955);
\begin{scriptsize}
\draw [fill=blue] (1.,4.) circle (2.5pt);
\draw[color=blue] (0.66,4.12) node {$1$};
\draw [fill=blue] (3.,4.) circle (2.5pt);
\draw[color=blue] (2.78,3.79) node {$4$};
\draw [fill=blue] (4.,5.732050807568878) circle (2.5pt);
\draw[color=blue] (4.02,5.25) node {$9$};
\draw [fill=blue] (3.,7.464101615137758) circle (2.5pt);
\draw[color=blue] (3.3,7.35) node {$12$};
\draw [fill=blue] (1.,7.464101615137759) circle (2.5pt);
\draw[color=blue] (1.32,7.75) node {$5$};
\draw [fill=blue] (0.,5.732050807568879) circle (2.5pt);
\draw[color=blue] (-0.1,6.17) node {$8$};
\draw [fill=blue] (4.00469641476435,9.193436698311338) circle (2.5pt);
\draw[color=blue] (4.14,9.57) node {$16$};
\draw [fill=blue] (0.008059480959714316,9.200793082314224) circle (2.5pt);
\draw[color=blue] (-0.26,9.53) node {$17$};
\draw [fill=blue] (-1.9999783181369541,5.7188805974618955) circle (2.5pt);
\draw[color=blue] (-2.32,6.05) node {$18$};
\draw [fill=blue] (0.07364803520534302,2.2264598291842406) circle (2.5pt);
\draw[color=blue] (-0.22,2.09) node {$13$};
\draw [fill=blue] (4.037474256249794,2.2898569351788662) circle (2.5pt);
\draw[color=blue] (4.18,2.07) node {$14$};
\draw [fill=blue] (6.,5.7320508075688785) circle (2.5pt);
\draw[color=blue] (6.28,5.79) node {$15$};
\draw[color=black] (1.05,8.56) node {$2$};
\draw[color=black] (-1.01,6.24) node {$2$};
\draw[color=black] (0.19,3.5) node {$2$};
\draw[color=black] (3.09,3.14) node {$2$};
\draw[color=black] (4.81,5.14) node {$2$};
\draw[color=black] (3.85,8.08) node {$2$};
\end{scriptsize}
\end{tikzpicture}
\begin{minipage}[b]{.35\textwidth}
\begin{tikzpicture}[line cap=round,line join=round,>=triangle 45,x=1.0cm,y=1.0cm,scale = .8]
\draw [shift={(-0.155,-5.927684547533456)},line width=.8pt,->]  plot[domain=1.0471975511965976:2.0243951023931953,variable=\t]({1.*5.69*cos(\t r)+0.*5.69*sin(\t r)},{0.*5.69*cos(\t r)+1.*5.69*sin(\t r)});
\draw [shift={(-0.155,3.927684547533456)},line width=.8pt,-> ]  plot[domain=4.1887902047863905:5.235987755982988,variable=\t]({1.*5.69*cos(\t r)+0.*5.69*sin(\t r)},{0.*5.69*cos(\t r)+1.*5.69*sin(\t r)});
\draw [shift={(5.535,3.9276845475334534)},line width=.8pt,->]  plot[domain=3.1415926535897927:4.12879020478639,variable=\t]({1.*5.69*cos(\t r)+0.*5.69*sin(\t r)},{0.*5.69*cos(\t r)+1.*5.69*sin(\t r)});
\draw [shift={(-3.,-1.)},line width=.8pt,->]  plot[domain=0.:1.0471975511965974,variable=\t]({1.*5.69*cos(\t r)+0.*5.69*sin(\t r)},{0.*5.69*cos(\t r)+1.*5.69*sin(\t r)});
\draw [shift={(-5.845,3.9276845475334587)},line width=.8pt,->]  plot[domain=-1.0471975511965983:-0.06,variable=\t]({1.*5.69*cos(\t r)+0.*5.69*sin(\t r)},{0.*5.69*cos(\t r)+1.*5.69*sin(\t r)});
\draw [shift={(2.69,-1.)},line width=.8pt,->]  plot[domain=2.0943951023931953:3.141592653589793,variable=\t]({1.*5.69*cos(\t r)+0.*5.69*sin(\t r)},{0.*5.69*cos(\t r)+1.*5.69*sin(\t r)});
\begin{scriptsize}
\draw [fill=blue] (-3.,-1.) circle (2.5pt);
\draw[color=blue] (-3.38,-1.15) node {$7$};
\draw [fill=blue] (2.69,-1.) circle (2.5pt);
\draw[color=blue] (2.96,-1.07) node {$11$};
\draw [fill=blue] (-0.155,3.927684547533456) circle (2.5pt);
\draw[color=blue] (-0.22,4.37) node {$3$};
\draw[color=black] (-0.22,-0.37) node {$-i$};
\draw[color=black] (-0.16,-1.93) node {$i$};
\draw[color=black] (0.78,1.39) node {$-i$};
\draw[color=black] (2.1,2.15) node {$i$};
\draw[color=black] (-1.02,1.47) node {$-i$};
\draw[color=black] (-2.47,2.36) node {$i$};
\end{scriptsize}
\end{tikzpicture}

\begin{tikzpicture}[line cap=round,line join=round,>=triangle 45,x=1.0cm,y=1.0cm,scale = .8]
\draw [shift={(-0.155,-5.927684547533456)},line width=.8pt,->]  plot[domain=1.0471975511965976:2.0243951023931953,variable=\t]({1.*5.69*cos(\t r)+0.*5.69*sin(\t r)},{0.*5.69*cos(\t r)+1.*5.69*sin(\t r)});
\draw [shift={(-0.155,3.927684547533456)},line width=.8pt,-> ]  plot[domain=4.1887902047863905:5.235987755982988,variable=\t]({1.*5.69*cos(\t r)+0.*5.69*sin(\t r)},{0.*5.69*cos(\t r)+1.*5.69*sin(\t r)});
\draw [shift={(5.535,3.9276845475334534)},line width=.8pt,->]  plot[domain=3.1415926535897927:4.12879020478639,variable=\t]({1.*5.69*cos(\t r)+0.*5.69*sin(\t r)},{0.*5.69*cos(\t r)+1.*5.69*sin(\t r)});
\draw [shift={(-3.,-1.)},line width=.8pt,->]  plot[domain=0.:1.0471975511965974,variable=\t]({1.*5.69*cos(\t r)+0.*5.69*sin(\t r)},{0.*5.69*cos(\t r)+1.*5.69*sin(\t r)});
\draw [shift={(-5.845,3.9276845475334587)},line width=.8pt,->]  plot[domain=-1.0471975511965983:-0.06,variable=\t]({1.*5.69*cos(\t r)+0.*5.69*sin(\t r)},{0.*5.69*cos(\t r)+1.*5.69*sin(\t r)});
\draw [shift={(2.69,-1.)},line width=.8pt,->]  plot[domain=2.0943951023931953:3.141592653589793,variable=\t]({1.*5.69*cos(\t r)+0.*5.69*sin(\t r)},{0.*5.69*cos(\t r)+1.*5.69*sin(\t r)});
\begin{scriptsize}
\draw [fill=blue] (-3.,-1.) circle (2.5pt);
\draw[color=blue] (-3.38,-1.15) node {$10$};
\draw [fill=blue] (2.69,-1.) circle (2.5pt);
\draw[color=blue] (2.96,-1.07) node {$2$};
\draw [fill=blue] (-0.155,3.927684547533456) circle (2.5pt);
\draw[color=blue] (-0.22,4.37) node {$6$};
\draw[color=black] (-0.22,-0.37) node {$-i$};
\draw[color=black] (-0.16,-1.93) node {$i$};
\draw[color=black] (0.78,1.39) node {$-i$};
\draw[color=black] (2.1,2.15) node {$i$};
\draw[color=black] (-1.02,1.47) node {$-i$};
\draw[color=black] (-2.47,2.36) node {$i$};
\end{scriptsize}
\end{tikzpicture}

\end{minipage}}
\end{center}
\caption{The decomposed graph $G$ after one round of decomposing using automorphism $\psi_0 = (1,4,7,10)$ $(2,5,8,11)(3,6,9,12)(13,16)(14,17)(15,18)$. The weights of unidirectional and bidirectional edges are equal to one unless otherwise stated. }\label{fig:ex32}
\end{figure}
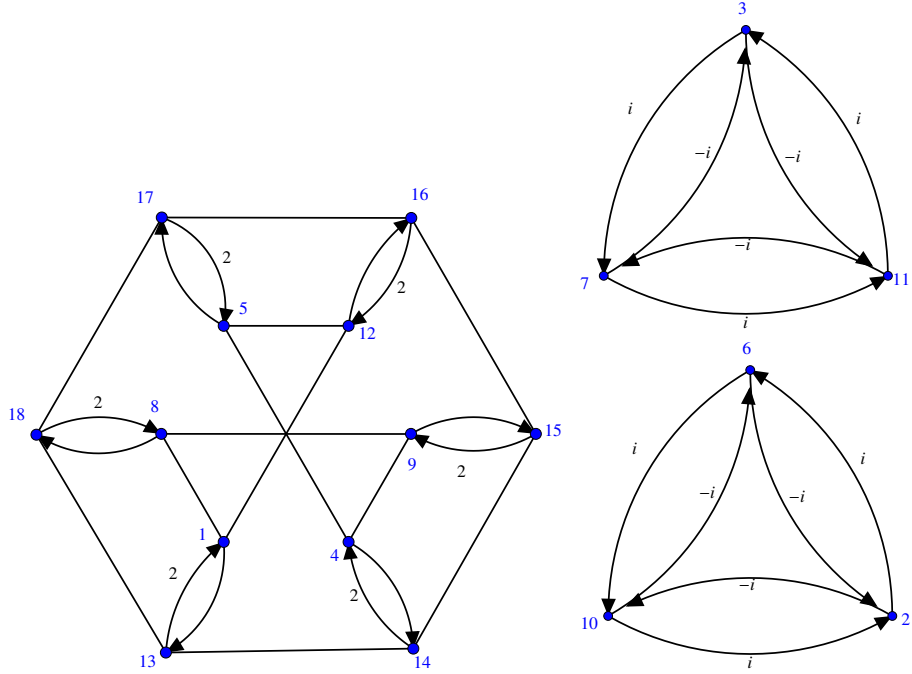
 The associated decomposed graph is shown in Figure \ref{fig:ex32}. In the second iteration of the sub-algorithm, we use Equation \eqref{eq:6} to find the new automorphism
$$\psi_{1}= \psi_0^2=(1,4)(5,8)(9,12)(13,16)(14,17)(15,18),$$
which has six cycles of length two. We now must choose the semi-transversals for the next step. We begin with $\cT_0 = \{1,5,9\}$ from the previous round. Now we are free to choose any element from the orbits of length $2^1$ of $\phi_{0}$, so we choose $13$, and then add $\{13, \phi ^{2}_0(13),\phi ^{2\cdot 2}_0(13)\}=\{13, 15,17\}$ to $\cT_0$.  Thus we have
\[
\cT_0=\{1,5,9,13,15,17\}\text{ and }\cT_1=\{4,8,12,16,18,14\}.
\]
 Using these transversals, the relevant block matrices are
 \[
 F = H = L  = \emptyset,\ C_0 = \left[\begin{array}{rrrrrr}
 0 & 0 & 0 & 1 & 0 & 0 \\
 0 & 0 & 0 & 0 & 0 & 1 \\
 0 & 0 & 0 & 0 & 1 & 0 \\
 2 & 0 & 0 & 0 & 0 & 0 \\
 0 & 0 & 2 & 0 & 0 & 0 \\
 0 & 2 & 0 & 0 & 0 & 0 \\
 \end{array}\right], \
 C_1 = \left[\begin{array}{rrrrrr}
 0 & 1 & 1 & 0 & 0 & 0 \\
 1 & 0 & 1 & 0 & 0 & 0 \\
 1 & 1 & 0 & 0 & 0 & 0 \\
 0 & 0 & 0 & 0 & 1 & 1 \\
 0 & 0 & 0 & 1 & 0 & 1 \\
 0 & 0 & 0 & 1 & 1 & 0 \\
 \end{array}\right].
 \]
 So, the second round results in an adjacency matrix $M_2 = \tM_2 \oplus \hM_2 \oplus \hM_1$ where
 \begin{equation}\label{eq:afterround1}
\tM_2 =  \left[\begin{array}{rrrrrr}
 0 & 1 & 1 & 1 & 0 & 0 \\
 1 & 0 & 1 & 0 & 0 & 1 \\
 1 & 1 & 0 & 0 & 1 & 0 \\
 2&  0 & 0 & 0 & 1 & 1 \\
 0 & 0 & 2 & 1 & 0 & 1 \\
 0 & 2 & 0 & 1 & 1 & 0 \\
 \end{array}\right] , \text{ and }
 \hM_2 =
  \left[\begin{array}{rrrrrr}
 0 & -1 & -1 & 1 & 0 & 0 \\
 -1 & 0 & -1 & 0 & 0 & 1 \\
 -1 & -1 & 0 & 0 & 1 & 0 \\
 2&  0 & 0 & 0 & -1 & -1 \\
 0 & 0 & 2 & -1 & 0 & -1 \\
 0 & 2 & 0 & -1 & -1 & 0 \\
 \end{array}\right]
 \end{equation}
 with the associated decomposed graphs given in Figure \ref{fig:afterround2}.

 \begin{figure}[h!]
 \begin{center}
 \resizebox{.7\textwidth}{!}{
\begin{minipage}{.58\textwidth}
\begin{tikzpicture}[line cap=round,line join=round,>=triangle 45,x=1.0cm,y=1.0cm]
\draw [line width=.8pt] (1.,3.)-- (3.,3.);
\draw [line width=.8pt] (3.,3.)-- (2.,1.2679491924311226);
\draw [line width=.8pt] (2.,1.2679491924311226)-- (1.,3.);
\draw [line width=.8pt] (-0.9543812530264946,4.15071705839042)-- (4.950264550997168,4.157751888343405);
\draw [line width=.8pt] (4.950264550997168,4.157751888343405)-- (2.,-1.);
\draw [line width=.8pt] (2.,-1.)-- (-0.9543812530264946,4.15071705839042);
\draw [->,shift={(4.982142282112744,1.8825401936775616)},line width=.8pt]  plot[domain=1.5848062958529314:2.6282437728220827,variable=\t]({1.*2.2754350013316795*cos(\t r)+0.*2.2754350013316795*sin(\t r)},{0.*2.2754350013316795*cos(\t r)+1.*2.2754350013316795*sin(\t r)});
\draw [->,shift={(2.9950278882874226,5.229888446850309)},line width=.8pt]  plot[domain=4.71461873473831:5.68161664111629,variable=\t]({1.*2.229893990146475*cos(\t r)+0.*2.229893990146475*sin(\t r)},{0.*2.229893990146475*cos(\t r)+1.*2.229893990146475*sin(\t r)});
\draw [->,shift={(-0.8384002308339037,2.112677591828245)},line width=.8pt]  plot[domain=0.4496796966608705:1.5276431464586627,variable=\t]({1.*2.041336930732865*cos(\t r)+0.*2.041336930732865*sin(\t r)},{0.*2.041336930732865*cos(\t r)+1.*2.041336930732865*sin(\t r)});
\draw [->,shift={(1.012838435363215,5.256826814013395)},line width=.8pt]  plot[domain=3.6538078194803316:4.706700330818788,variable=\t]({1.*2.256863330791748*cos(\t r)+0.*2.256863330791748*sin(\t r)},{0.*2.256863330791748*cos(\t r)+1.*2.256863330791748*sin(\t r)});
\draw [->,shift={(3.9912401861237243,0.13397459621556207)},line width=.8pt]  plot[domain=2.6239155065362585:3.6592698006433286,variable=\t]({1.*2.291492060578934*cos(\t r)+0.*2.291492060578934*sin(\t r)},{0.*2.291492060578934*cos(\t r)+1.*2.291492060578934*sin(\t r)});
\draw [->,shift={(0.1683484396646991,0.13397459621556054)},line width=.8pt]  plot[domain=-0.5543450149577236:0.554345014957724,variable=\t]({1.*2.154262245721488*cos(\t r)+0.*2.154262245721488*sin(\t r)},{0.*2.154262245721488*cos(\t r)+1.*2.154262245721488*sin(\t r)});
\begin{scriptsize}
\draw [fill=blue] (3.,3.) circle (2.5pt);
\draw[color=blue] (2.89,3.3) node {$1$};

\draw [fill=blue] (1.,3.) circle (2.5pt);
\draw[color=blue] (0.86,2.79) node {$5$};

\draw [fill=blue] (2.,1.2679491924311226) circle (2.5pt);
\draw[color=blue] (2.3,1.37) node {$9$};

\draw [fill=blue] (-0.9543812530264946,4.15071705839042) circle (2.5pt);
\draw[color=blue] (-1.1,4.39) node {$17$};

\draw [fill=blue] (2.,-1.) circle (2.5pt);
\draw[color=blue] (1.98,-1.23) node {$15$};

\draw [fill=blue] (4.950264550997168,4.157751888343405) circle (2.5pt);
\draw[color=blue] (5.2,4.23) node {$13$};

\draw[color=black] (3.6,3.85) node {$2$};
\draw[color=black] (0.07,2.94) node {$2$};
\draw[color=black] (2.47,0.56) node {$2$};
\end{scriptsize}
\end{tikzpicture}

\begin{tikzpicture}[line cap=round,line join=round,>=triangle 45,x=1.0cm,y=1.0cm]
\draw [line width=.8pt] (1.,3.)-- (3.,3.);
\draw [line width=.8pt] (3.,3.)-- (2.,1.2679491924311226);
\draw [line width=.8pt] (2.,1.2679491924311226)-- (1.,3.);
\draw [line width=.8pt] (-0.9543812530264946,4.15071705839042)-- (4.950264550997168,4.157751888343405);
\draw [line width=.8pt] (4.950264550997168,4.157751888343405)-- (2.,-1.);
\draw [line width=.8pt] (2.,-1.)-- (-0.9543812530264946,4.15071705839042);
\draw [->,shift={(4.982142282112744,1.8825401936775616)},line width=.8pt]  plot[domain=1.5848062958529314:2.6282437728220827,variable=\t]({1.*2.2754350013316795*cos(\t r)+0.*2.2754350013316795*sin(\t r)},{0.*2.2754350013316795*cos(\t r)+1.*2.2754350013316795*sin(\t r)});
\draw [->,shift={(2.9950278882874226,5.229888446850309)},line width=.8pt]  plot[domain=4.71461873473831:5.68161664111629,variable=\t]({1.*2.229893990146475*cos(\t r)+0.*2.229893990146475*sin(\t r)},{0.*2.229893990146475*cos(\t r)+1.*2.229893990146475*sin(\t r)});
\draw [->,shift={(-0.8384002308339037,2.112677591828245)},line width=.8pt]  plot[domain=0.4496796966608705:1.5276431464586627,variable=\t]({1.*2.041336930732865*cos(\t r)+0.*2.041336930732865*sin(\t r)},{0.*2.041336930732865*cos(\t r)+1.*2.041336930732865*sin(\t r)});
\draw [->,shift={(1.012838435363215,5.256826814013395)},line width=.8pt]  plot[domain=3.6538078194803316:4.706700330818788,variable=\t]({1.*2.256863330791748*cos(\t r)+0.*2.256863330791748*sin(\t r)},{0.*2.256863330791748*cos(\t r)+1.*2.256863330791748*sin(\t r)});
\draw [->,shift={(3.9912401861237243,0.13397459621556207)},line width=.8pt]  plot[domain=2.6239155065362585:3.5592698006433286,variable=\t]({1.*2.291492060578934*cos(\t r)+0.*2.291492060578934*sin(\t r)},{0.*2.291492060578934*cos(\t r)+1.*2.291492060578934*sin(\t r)});
\draw [->,shift={(0.1683484396646991,0.13397459621556054)},line width=.8pt]  plot[domain=-0.5543450149577236:0.554345014957724,variable=\t]({1.*2.154262245721488*cos(\t r)+0.*2.154262245721488*sin(\t r)},{0.*2.154262245721488*cos(\t r)+1.*2.154262245721488*sin(\t r)});
\begin{scriptsize}
\draw [fill=blue] (3.,3.) circle (2.5pt);
\draw[color=blue] (2.89,3.3) node {$4$};

\draw [fill=blue] (1.,3.) circle (2.5pt);
\draw[color=blue] (0.86,2.79) node {$8$};

\draw [fill=blue] (2.,1.2679491924311226) circle (2.5pt);
\draw[color=blue] (2.3,1.37) node {$12$};

\draw [fill=blue] (-0.9543812530264946,4.15071705839042) circle (2.5pt);
\draw[color=blue] (-1.1,4.39) node {$14$};

\draw [fill=blue] (2.,-1.) circle (2.5pt);
\draw[color=blue] (1.98,-1.23) node {$18$};

\draw [fill=blue] (4.950264550997168,4.157751888343405) circle (2.5pt);
\draw[color=blue] (5.16,4.33) node {$16$};

\draw[color=black] (2.02,3.21) node {$-1$};
\draw[color=black] (2.6,2.11) node {$-1$};
\draw[color=black] (1.3,2.15) node {$-1$};
\draw[color=black] (2.,4.41) node {$-1$};
\draw[color=black] (3.74,1.35) node {$-1$};
\draw[color=black] (0.26,1.55) node {$-1$};
\draw[color=black] (3.6,4.03) node {$2$};
\draw[color=black] (0.07,3.14) node {$2$};
\draw[color=black] (2.47,0.56) node {$2$};
\end{scriptsize}
\end{tikzpicture}
 \end{minipage}
 \begin{minipage}{.4\textwidth}
 \begin{tikzpicture}[line cap=round,line join=round,>=triangle 45,x=1.0cm,y=1.0cm,scale = .8]
\draw [shift={(-0.155,-5.927684547533456)},line width=.8pt,->]  plot[domain=1.0471975511965976:2.0243951023931953,variable=\t]({1.*5.69*cos(\t r)+0.*5.69*sin(\t r)},{0.*5.69*cos(\t r)+1.*5.69*sin(\t r)});
\draw [shift={(-0.155,3.927684547533456)},line width=.8pt,-> ]  plot[domain=4.1887902047863905:5.235987755982988,variable=\t]({1.*5.69*cos(\t r)+0.*5.69*sin(\t r)},{0.*5.69*cos(\t r)+1.*5.69*sin(\t r)});
\draw [shift={(5.535,3.9276845475334534)},line width=.8pt,->]  plot[domain=3.1415926535897927:4.12879020478639,variable=\t]({1.*5.69*cos(\t r)+0.*5.69*sin(\t r)},{0.*5.69*cos(\t r)+1.*5.69*sin(\t r)});
\draw [shift={(-3.,-1.)},line width=.8pt,->]  plot[domain=0.:1.0471975511965974,variable=\t]({1.*5.69*cos(\t r)+0.*5.69*sin(\t r)},{0.*5.69*cos(\t r)+1.*5.69*sin(\t r)});
\draw [shift={(-5.845,3.9276845475334587)},line width=.8pt,->]  plot[domain=-1.0471975511965983:-0.06,variable=\t]({1.*5.69*cos(\t r)+0.*5.69*sin(\t r)},{0.*5.69*cos(\t r)+1.*5.69*sin(\t r)});
\draw [shift={(2.69,-1.)},line width=.8pt,->]  plot[domain=2.0943951023931953:3.141592653589793,variable=\t]({1.*5.69*cos(\t r)+0.*5.69*sin(\t r)},{0.*5.69*cos(\t r)+1.*5.69*sin(\t r)});
\begin{scriptsize}
\draw [fill=blue] (-3.,-1.) circle (2.5pt);
\draw[color=blue] (-3.38,-1.15) node {$7$};
\draw [fill=blue] (2.69,-1.) circle (2.5pt);
\draw[color=blue] (2.96,-1.07) node {$11$};
\draw [fill=blue] (-0.155,3.927684547533456) circle (2.5pt);
\draw[color=blue] (-0.22,4.37) node {$3$};
\draw[color=black] (-0.22,-0.37) node {$-i$};
\draw[color=black] (-0.16,-1.93) node {$i$};
\draw[color=black] (0.78,1.39) node {$-i$};
\draw[color=black] (2.1,2.15) node {$i$};
\draw[color=black] (-1.02,1.47) node {$-i$};
\draw[color=black] (-2.47,2.36) node {$i$};
\end{scriptsize}
\end{tikzpicture}

\begin{tikzpicture}[line cap=round,line join=round,>=triangle 45,x=1.0cm,y=1.0cm,scale = .8]
\draw [shift={(-0.155,-5.927684547533456)},line width=.8pt,->]  plot[domain=1.0471975511965976:2.0243951023931953,variable=\t]({1.*5.69*cos(\t r)+0.*5.69*sin(\t r)},{0.*5.69*cos(\t r)+1.*5.69*sin(\t r)});
\draw [shift={(-0.155,3.927684547533456)},line width=.8pt,-> ]  plot[domain=4.1887902047863905:5.235987755982988,variable=\t]({1.*5.69*cos(\t r)+0.*5.69*sin(\t r)},{0.*5.69*cos(\t r)+1.*5.69*sin(\t r)});
\draw [shift={(5.535,3.9276845475334534)},line width=.8pt,->]  plot[domain=3.1415926535897927:4.12879020478639,variable=\t]({1.*5.69*cos(\t r)+0.*5.69*sin(\t r)},{0.*5.69*cos(\t r)+1.*5.69*sin(\t r)});
\draw [shift={(-3.,-1.)},line width=.8pt,->]  plot[domain=0.:1.0471975511965974,variable=\t]({1.*5.69*cos(\t r)+0.*5.69*sin(\t r)},{0.*5.69*cos(\t r)+1.*5.69*sin(\t r)});
\draw [shift={(-5.845,3.9276845475334587)},line width=.8pt,->]  plot[domain=-1.0471975511965983:-0.06,variable=\t]({1.*5.69*cos(\t r)+0.*5.69*sin(\t r)},{0.*5.69*cos(\t r)+1.*5.69*sin(\t r)});
\draw [shift={(2.69,-1.)},line width=.8pt,->]  plot[domain=2.0943951023931953:3.141592653589793,variable=\t]({1.*5.69*cos(\t r)+0.*5.69*sin(\t r)},{0.*5.69*cos(\t r)+1.*5.69*sin(\t r)});
\begin{scriptsize}
\draw [fill=blue] (-3.,-1.) circle (2.5pt);
\draw[color=blue] (-3.38,-1.15) node {$10$};
\draw [fill=blue] (2.69,-1.) circle (2.5pt);
\draw[color=blue] (2.96,-1.07) node {$2$};
\draw [fill=blue] (-0.155,3.927684547533456) circle (2.5pt);
\draw[color=blue] (-0.22,4.37) node {$6$};
\draw[color=black] (-0.22,-0.37) node {$-i$};
\draw[color=black] (-0.16,-1.93) node {$i$};
\draw[color=black] (0.78,1.39) node {$-i$};
\draw[color=black] (2.1,2.15) node {$i$};
\draw[color=black] (-1.02,1.47) node {$-i$};
\draw[color=black] (-2.47,2.36) node {$i$};
\end{scriptsize}
\end{tikzpicture}

 \end{minipage}
 }
 \end{center}
 \caption{The decomposed graph $G$ after the second round using $\psi_0 = (1,4,7,10)$ $(2,5,8,11)(3,6,9,12)(13,16)(14,17)(15,18)$.The weights of unidirectional and bidirectional edges are equal to one unless otherwise stated. }\label{fig:afterround2}
 \end{figure}
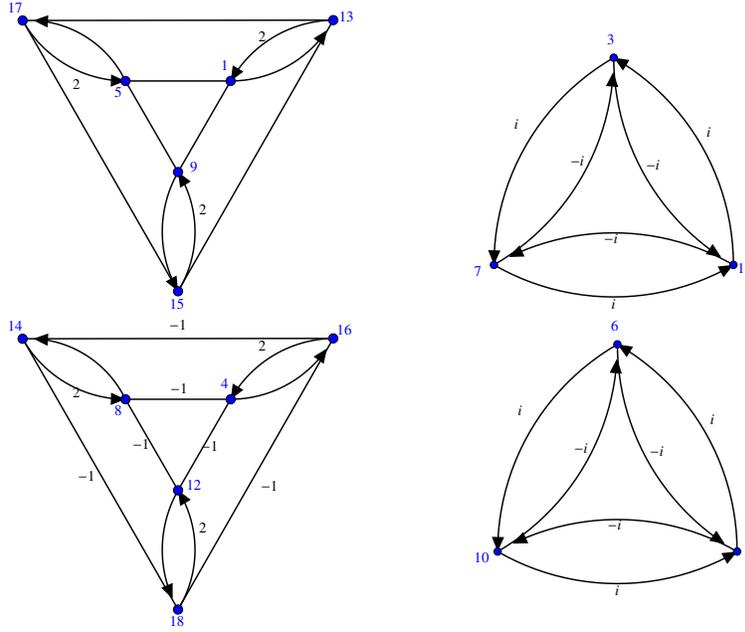


\noindent\emph{Step C}: Now we have exhausted the decompositions that can be accomplished using $p_0 = 2$.  We move to the prime $p_1 = 3$, and notice that
\[
1 = 1 \cdot 4 + (-1)\cdot 3.
\]
Thus we now use the automorphism
\[
\phi_{1}= \tp = \phi_0^{(1+3)}= (1,5,9)(2,6,10)(3,7,11)(4,8,12)(13,17,15)(14,18,16).
\]
\begin{center}
\emph{Round II}
\end{center}
\noindent\emph{Step A}: Here $\ell_2=1$ and $\psi_1=\phi_1^1$. We begin with the matrix $M_0 $ equal to $M_2$ from the previous round. \\
\noindent \emph{Step B}: On this step we only need to run through the above algorithm once since the order of $\phi_1$ is three. It is worth mentioning that during this step we do not need to be worried about how we choose the transversal for this decomposition because this is the final step and the transversal only needs to be carefully chosen to guarantee that the resulting decomposed graph contains a symmetry for the next round. We choose the transversal $\cT_0 = \{1,2,3,4,13,14\}$, thus
\[
\cT_0 = \{1,13,4,14, 2, 3\}, \ \cT_1 = \{5,17,8,18,6,7\}, \text{ and }
\cT_2 = \{9,15,12,16, 10,11\}.
\]  The relevant block matrices for this decomposition are
\[
\begin{split}
C_0 = & \left[\begin{array}{rr|rr|r|r}
0 &  1 & 0 & 0 & 0 & 0 \\
2 &  0 & 0 & 0 & 0 & 0\\\hline
0 &  0 & 0 & 0 & 0 & 0 \\
0 &  0 & 0 & 0 & 0 & 0\\\hline
 0 & 0& 0 & 0& 0 & 0\\\hline
  0 & 0& 0 & 0& 0 & 0
\end{array}\right], \
C_1 = \left[\begin{array}{rr|rr|r|r}
1 & 0 & 0 & 0 & 0 & 0 \\
0 & 1 & 0 & 0 & 0 & 0 \\\hline
0 & 0 & -1 & 0 & 0 & 0 \\
0 & 0 & 2 & -1 & 0 & 0 \\\hline
0  & 0 & 0 & 0 & i & 0\\\hline
0  & 0 & 0 & 0& 0 & i \\
\end{array}\right], \\&
\qquad C_2 = \left[\begin{array}{rr|rr|r|r}
1 & 0 & 0 & 0 & 0 & 0 \\
0 & 1 & 0 & 0 & 0 & 0 \\\hline
0 & 0 & -1 & 1& 0 & 0  \\
0 & 0 & 0 & -1& 0 & 0  \\\hline
0 & 0 & 0 & 0& -i & 0  \\\hline
0 & 0 & 0 & 0& 0 & -i  \\
\end{array}\right].
\end{split}
\]
So, setting $\omega = e^{\frac{2 \pi i }{3}}$ the final matrix decomposition is $M_1 = \tM_1 \oplus \hM_1$ with
\[
 \begin{array}{c}
\tM_1 =\left[
\begin{array}{rr}2 & 1 \\ 2 & 2\end{array}\right] \oplus
\left[ \begin{array}{rr} -2 & 1 \\ 2 & -2 \end{array}\right]  \oplus [0] \oplus [ 0] \\[4mm]
\hM_1 =
 \left[\begin{array}{rr}-1& 1 \\ 2 & -1\end{array}\right]\oplus
\left[ \begin{array}{rr} 1 & \omega^2 \\ 2\omega &1 \end{array} \right] \oplus [ \sqrt{3}]\oplus[ \sqrt{3} ]\oplus\cdots \\[4mm]
 \qquad \qquad \left[\begin{array}{rr}-1& 1 \\ 2 & -1\end{array}\right]\oplus
\left[\begin{array}{rr} 1 & \omega \\ 2\omega^2 &1 \end{array} \right] \oplus[- \sqrt{3}]\oplus[ -\sqrt{3} ]
\end{array}
\]


\noindent \emph{Step C}: There is no need to find $\phi_2$ since we cannot decompose this matrix any further.\\[2mm]

\noindent Notice that the first $2 \times 2$ matrix appearing in the presentation of $\tM_1$ above 
 is precisely the divisor matrix associated with the original automorphism $\phi$.

\end{example}


\section{Spectral Radius of the divisor Matrix}\label{sec:extapp}
One particularly useful spectral property is the spectral radius associated with the graph structure $G$ of a network. The \emph{spectral radius}\index{spectral radius} of a matrix $M$ associated with $G$ is given by
\[\rho(M)=\max\{|\lambda|:\lambda\in\sigma(M)\}.\]
The spectral radius $\rho(M)$ of a network, {or more generally a dynamical system}, is particularly important for studying the system's dynamics. For instance, the matrix $M$ associated with a network may be a global or local linearization of the system of equations that govern the network's dynamics. If the network's dynamics are modeled by a discrete-time system, then stability of the system is guaranteed if $\rho(M)<1$ and local instability results when $\rho(M)>1$ \cite{HA03}.

From the theory of equitable partitions it follows that $\sigma(M_{\phi})\subset\sigma(M)$. In fact, it was shown in \cite{FSSW} that the spectral radius $\rho(M)$ of $M$ is an eigenvalue of $M_{\phi}$ if $\phi$ is separable and the matrix $M$ is nonnegative and irreducible. 
In this section we will generalize the result to to show that the spectral radius $\rho(M)$ of an automorphism compatible matrix $M$ is the same as the spectral radius of the divisor matrix $M_\phi$ for any automorphism $\phi$ if $M$ is both nonnegative and irreducible.

\begin{prop}\label{lem:3}\textbf{(Spectral Radius of a General Equitable Partition)}
Let $\phi$ be any automorphism of a graph $G$ with $M$ an automorphism compatible matrix. If $M$ is nonnegative and irreducible, then $\rho(M)=\rho(M_{\phi})$.
\end{prop}
Before we can prove the proposition we need the following lemma.
\begin{lem}\label{lem:bc}
If an irreducible, nonnegative matrix $A$ has block circulant form
$$A=\left[\begin{matrix}
A_1 & A_2 & A_3 & \dots & A_n \\
A_n & A_1 & A_2 & \dots & A_{n-1}\\
A_{n-1} & A_n & A_1 & \dots & A_{n-2}\\
\vdots & \vdots & \vdots & \ddots & \vdots\\
A_2 & A_3 & A_4 & \dots & A_1\end{matrix}\right]\text{ and   } B=\sum_{m=1}^n A_m, $$
then $\rho(A)=\rho(B)$.
\end{lem}
\begin{proof}
Because A is nonnegative and irreducible, then $B$ must also be nonnegative and irreducible.  Thus the Perron-Frobenius Theorem guarantees that the spectral radius of B, $\rho(B)$, is a positive eigenvalue of $B$.  It also guarantees that  the eigenvector $\mathbf{v}$ associated to $\rho(B)$ can be chosen to have all positive entries.  Now consider the vector $\mathbf{w}=\mathbf{v}\oplus\mathbf{v}\oplus\dots \oplus\mathbf{v}$ (a total of $m$ $\mathbf{v}$'s in the direct sum).  We can see that $\mathbf{w}$ is an eigenvector of $A$ since $$A\mathbf{w}=\left[\begin{matrix}
A_1 & A_2 & A_3 & \dots & A_n \\
A_n & A_1 & A_2 & \dots & A_{n-1}\\
A_{n-1} & A_n & A_1 & \dots & A_{n-2}\\
\vdots & \vdots & \vdots & \ddots & \vdots\\
A_2 & A_3 & A_4 & \dots & A_1\end{matrix}\right]\left[\begin{matrix} \mathbf{v} \\ \mathbf{v}\\ \mathbf{v}\\ \vdots \\ \mathbf{v}\end{matrix}\right]=\left[\begin{matrix} \sum A_m \mathbf{v} \\\sum A_m  \mathbf{v}\\ \sum A_m \mathbf{v}\\ \vdots \\ \sum A_m \mathbf{v}\end{matrix}\right]=\left[\begin{matrix} \rho(N) \mathbf{v} \\\rho(N)  \mathbf{v}\\\rho(N) \mathbf{v}\\ \vdots \\ \rho(N) \mathbf{v}\end{matrix}\right]=\rho(N) \mathbf{w}.$$  Thus $\mathbf{w}$ is an eigenvector of $A$ with only positive entries and with eigenvalue $\rho(B)$.  Because $A$ is irreducible and nonnegative, the Perron-Frobenius Theorem tells us the only eigenvector of $A$ that is all positive must correspond the the largest eigenvalue, which is the spectral radius.  Therefore we conclude that $\rho(B)=\rho(A).$
\end{proof}
Now we prove Proposition \ref{lem:3}.
\begin{proof}

Suppose we have the matrix $M$ which is irreducible and nonnegative with an associated automorphism $\phi$. Using the process outlined in the algorithm found in Section \ref{GED}, we can decompose $M$ through a process of similarity transformations.  Using the notation found in \ref{lem:Dallas},  we will first show after the similarity transformation in this lemma that $\rho(M)=\rho(\tilde{M})$.  To equitably decompose a matrix completely, we repeatedly do this similarity transformation. At the final step, $\tilde{M}=M_\phi$, the divisor matrix associated with $\phi$.

In order to perform the similarity transformation, we first need to reorder the rows and columns of $M$ as prescribed in Proposition \ref{prop:primepower} and construct the $T$ matrix. From Lemma \ref{lem:Dallas}, we have

 \[
T^{-1} M T = \tilde{M}
\oplus B_1 \oplus B_2 \oplus \cdots \oplus B_{p^N-p^{N-1}},
\]
where $\tilde{M}=\left[\begin{array}{rr} F & kH \\ L & B_0 \end{array}\right]$. First we will prove the claim that $\rho(\tilde{M})\geq\rho(B_j)$ for $1\leq j\leq p^{N}-p^{N-1}.$

To begin, we use Corollary 8.1.20 in \cite{Horn85} which states for a nonnegative matrix $Q$, if $P$ is a principal submatrix of $Q$ then $\rho(P)\leq\rho(Q)$. Since $M$ is nonnegative, and $B_0=\sum_{m=0}^{p-1}D_m$, we know $\tilde{M}$ is nonnegative.  Because $B_0$ is a principal submatrix of $\tilde{M}$, using the corollary we conclude that $\rho(B_0)\leq\rho(\tilde{M})$.

Next we must show that $\rho(B_j)\leq\rho(B_0)$ for $1\leq j \leq p^N-p^{N-1}$.
Because
$$\left| {{B_j}} \right| = \left| {\sum\limits_{m = 0}^{p^N - 1} {{{\left( {{\omega ^{m\gamma_j}}} \right)}}{C_m}} } \right| \leq \sum\limits_{m = 0}^{p^N - 1} {\left| {\left( {{\omega ^{m\gamma_j}}} \right){C_m}} \right|}  = \sum\limits_{m = 0}^{p^N - 1} {{C_m}}, $$
we can use Theorem 8.1.18 in \cite{Horn85} to conclude that $$\rho(B_j)\leq\rho\left(\sum_{m=1}^{p^N-1}C_m\right)$$ for all $1\leq j\leq p^N-p^{N-1}$. We note that $$B_0=\sum_{i=0}^{p-1}D_i=
\left[\begin{matrix}
\sum_{t=0}^{p-1}C_{tp^{N-1}} & \sum_{t=0}^{p-1}C_{tp^{N-1}+1} & \dots & \sum_{t=0}^{p-1}C_{tp^{N-1}+p^N-1}\\
\sum_{t=0}^{p-1}C_{tp^{N-1}+p^N-1} & \sum_{t=0}^{p-1}C_{tp^{N-1}} & \dots & \sum_{t=0}^{p-1}C_{tp^{N-1}+p^N-2}\\
\vdots & \vdots & \ddots & \vdots\\
\sum_{t=0}^{p-1}C_{tp^{N-1}+1} & \sum_{t=0}^{p-1}C_{tp^{N-1}+2} & \dots & \sum_{t=0}^{p-1}C_{tp^{N-1}},
\end{matrix}\right]$$ which has block-circulant form.  Now we apply Lemma \ref{lem:bc}, which shows that $\rho(B_0)=\rho\left(\sum C_m\right)$. Therefore,
\begin{equation}\label{eq:rho}
\rho(B_j)\leq\rho(B_0)\leq\rho(\tilde{M}) \ \text{for all} \ 1\leq j\leq p^N-p^{N-1},
\end{equation}
which verifies our claim.

Using this claim and the fact that $\sigma(M)=\sigma(\tilde{M})\cup\sigma(B_1)\cup\dots\cup\sigma(B_{p^N-p^{N-1}})$ we can immediately conclude that $\rho(M)=\rho(\tilde{M})$.

The above argument is for a single step in the prime-power decomposition. To completely equitably decompose a matrix we are required to do a sequence of similarity transforms as described in section \ref{GED}.  Each similarity transform breaks the matrix into $M_i=\tilde{M}_i\oplus \hat{M}_i$ in the notation of Theorem \ref{thm:fullprimedecomp}. We just showed above that largest eigenvalue of $M_i$ is also the largest eigenvalue of $\tilde{M_i}$. However, in order to apply the above argument on each $M_i$, we just need to prove that if $M$ is nonnegative and irreducible then $\tilde{M}$ is also irreducible and nonnegative. If $M$ is nonnegative, $\tilde{M}$ is built from elements of $M$ and sums of elements from $M$ (see equation \ref{eq:mtilde}), thus $\tilde{M}$ is also nonnegative.  Also if $M$ is irreducible, then we claim $\tilde{M}$ must also be irreducible. This fact is proven in the proof of Proposition 4.3 in \cite{FSSW}.

If $\phi$ has order $p_1^{N_1} p_2^{N_2}\dots p_h^{N_h}$, then Proposition \ref{prop:dd} shows that it is possible to decompose the matrix using a sequence of automorphisms that induce a sequence of equitable decompositions on $M$. By induction each subsequent decomposition results in a nonnegative divisor matrix $(\dots(M_{\psi_0})_{\psi_1}\dots)_{\psi_i}$ for $i\leq h$ with the same spectral radius $r=\rho(M)$ implying that $\rho(M)$ must be the largest eigenvalue for $M_\phi$ for any $\phi\in Aut(G).$
\end{proof}

It is worth noting that many matrices typically associated with real networks are both nonnegative and irreducible. This includes the adjacency matrix as well as other weighted matrices \cite{Newman10}; although, there are some notable exceptions, including Laplacian matrices. Moreover, when analyzing the stability of a network, a linearization $M$ of the network's dynamics inherits the symmetries of the network's structure. Hence, if a symmetry of the network's structure is known then this symmetry can be used to decompose $M$ into a smaller divisor matrix $M_{\phi}$. As $M$ and $M_{\phi}$ have the same spectral radius, under the conditions stated in Proposition \ref{lem:3}, then one can use the smaller matrix $M_{\phi}$ to either calculate or estimate the spectral radius of the original unreduced network as is demonstrated in the following example.

\begin{example}
Returning to Example \ref{ex:18}, we can calculate that the eigenvalues of the graph's adjacency matrix $M$ in this example are
\[
\sigma(M)=\{ -2 \pm \sqrt{2}, 2 \pm \sqrt{2}, \pm1 \pm \sqrt{2}, 0 \}
\]
not including multiplicities. Similarly, we can compute that the eigenvalues of $M_{\phi}$ are
\[
\sigma(M_\phi)=\{2 + \sqrt{2}, 2 - \sqrt{2}\}.
\]
Thus, $\rho(M)=\rho(M_\phi)=2 + \sqrt{2}$.
 \end{example}
\section{Conclusion}\label{sec:conc}

The theory introduced in this paper extends the previous theory of equitable decompositions allowing one to decompose an automorphism compatible matrix over any of the associated graph's automorphisms. The result is a number of smaller matrices (or, equivalently, graphs) whose collective eigenvalues are the same as those associated with the original matrix (or graph). We note that the only restriction on equitably decomposing a matrix over an automorphism is that the matrix is automorphism compatible. That is, the matrix needs to respect the structure of the graph's group of automorphisms. Automorphism compatible matrices include the graph's weighted adjacency matrix, various Laplacian matrices, distance matrix, etc. Hence, a large number of matrices that are typically associated with a graph can be decomposed over any of the graph's automorphisms.

With this in mind, it is worth mentioning that although an equitable decomposition can be performed with respect to any automorphism there are equitable partitions that do not correspond to any graph automorphism. It is currently unknown whether there are any classes of matrices that can be decomposed with respect to these nonautomorphism based partitions.

Additionally, there are a few open questions regarding the algorithms that are introduced in this paper. For instance, in each of the algorithms given here there is a certain amount of freedom in how we perform an equitable decomposition (i.e., how we choose transversals at different stages in these decompositions). It is unknown to what extent the resulting equitable decomposition depends on the choice of transversals. It is likewise unknown how these different choices affect the computational complexity of these algorithms.

Finally, in a previous paper on equitable decompositions \cite{FSSW} it was shown that the eigenvalue approximation of Gershgorin improves as a matrix is equitably decomposed so long as the matrix is decomposed over either a basic or separable automorphism. It is an open question whether the same is true for the general equitable decompositions introduced here. Moreover, it is unknown if the related eigenvalue approximations of Brauer, Brualdi, and Varga improve under the process of equitable decomposition (see, for instance, \cite{Horn85,Varga}).

\section*{Acknowledgments}
This work was partially supported by the DoD grant HDTRA1-15- 0049.}

\bibliography{references}{}
\bibliographystyle{unsrt} 

\end{document}